\documentclass[12pt]{article}

\usepackage[margin=2cm]{geometry}
\usepackage{amsmath,amssymb,latexsym,amsthm}
\usepackage[all]{xy}

\theoremstyle{plain}
\newtheorem{proposition}{Proposition}[section]
\newtheorem{theorem}[proposition]{Theorem}
\newtheorem{corollary}[proposition]{Corollary}
\newtheorem{lemma}[proposition]{Lemma}

\newtheorem{definition}[proposition]{Definition}

\theoremstyle{definition}

\newcommand{\vnten}{{\overline{\otimes}}}
\newcommand{\ip}[2]{{\langle {#1} , {#2} \rangle}}
\newcommand{\mc}{\mathcal}
\newcommand{\mf}{\mathfrak}
\newcommand{\G}{\mathbb{G}}
\newcommand{\supp}{\operatorname{supp}}
\newcommand{\op}{{\operatorname{op}}}
\newcommand{\lin}{\operatorname{lin}}

\begin{document}

\title{Isometries between quantum convolution algebras}
\author{Matthew Daws, Hung Le Pham}

\maketitle

\begin{abstract}
Given locally compact quantum groups $\G_1$ and $\G_2$, we show that if the
convolution algebras $L^1(\G_1)$ and $L^1(\G_2)$ are isometrically isomorphic
as algebras, then $\G_1$ is isomorphic either to $\G_2$ or the commutant $\G_2'$.
Furthermore, given an isometric algebra isomorphism $\theta:L^1(\G_2)
\rightarrow L^1(\G_1)$, the adjoint is a $*$-isomorphism between $L^\infty(\G_1)$
and either $L^\infty(\G_2)$ or its commutant, composed with a twist given by a
member of the intrinsic group of $L^\infty(\G_2)$.  This extends known results
for Kac algebras (although our proofs are somewhat different)
which in turn generalised classical results of Wendel and
Walter.  We show that the same result holds for isometric algebra homomorphisms
between quantum measure algebras (either reduced or universal).  We make some
remarks about the intrinsic groups of the enveloping von Neumann algebras of
C$^*$-algebraic quantum groups.


\end{abstract}

\section{Introduction}

Locally compact quantum groups generalise Kac algebras, and form an
abstract generalisation of Pontryagin duality.  For a locally compact quantum
group $\G$, we shall write $L^\infty(\G)$ for the von Neumann algebraic quantum
group, and $C_0(\G)$ for the (reduced) C$^*$-algebraic quantum group.  As one
can move between these algebras, we tend to view them as representing the same
object $\G$.  Let $L^1(\G)$ be the ``quantum convolution algebra'', which is
the predual of $L^\infty(\G)$, made into a Banach algebra by using the
coproduct.  We can alternatively identify $L^1(\G)$ as a certain closed ideal
in $C_0(\G)^*$.  Notice that even in the classical case, where $G$ is even
an abelian locally compact group, the algebra $L^1(G)$ does not determine $G$,
as if $G$ is finite, then $L^1(G)$ is isomorphic to $C(\hat G)$, the continuous
functions on the dual group $\hat G$, and so $L^1(G)$ is isomorphic to $L^1(H)$
if and only if $\hat G$ and $\hat H$ are of the same cardinality.

However, Wendel's theorem \cite{wen} shows that if we take the norm into
account, then $L^1(G)$ completely determines $G$.  To be precise, if $\theta:
L^1(G_2)\rightarrow L^1(G_1)$ is an \emph{isometric} algebra isomorphism, then
there is a character $\chi$ on $G_1$, a positive constant $c>0$,
and a continuous group homomorphism $\alpha:G_1\rightarrow G_2$ such that
$\theta(f)(s) = c \chi(s) f(\alpha(s))$ almost everywhere for $s\in G_1$.
The constant $c$ simply reflects the fact that the Haar measure is only unique
up to a constant.  This was generalised to Fourier algebras by Walter,
\cite{wal}: here notice that $A(G)$ and $A(G^\op)$ are also isometrically
isomorphic, where $G^\op$ is the opposite group to $G$, and indeed Walter's
theorem shows (amongst other things) that $A(G_1)$ and $A(G_2)$ are
isometrically isomorphic if and only if $G_1$ is isomorphic to either $G_2$
or $G_2^\op$.

The Kac algebra case was shown by De Canni{\`e}re, Enock and Schwartz in
\cite{des} (see also \cite{es}).  The proof in the Kac algebra case uses that
the antipode is bounded, which is no longer true in the locally compact quantum
group case.  We
instead use a characterisation of the unitary antipode through the Haar weight
(see \cite[Proposition~5.20]{kv} and Section~\ref{sec:ua} below).
The intuitive idea is to show that an isometric algebra isomorphism must
intertwine the unitary antipode, although our actual argument is slightly
indirect.

Our principle result is that when $\theta:L^1(\G_2)\rightarrow L^1(\G_1)$ is
an isometric algebra isomorphism, then there is $u$, a member of the intrinsic
group of $L^\infty(\G_2)$, such that $x\mapsto \theta^*(x)u$ is either a
$*$-isomorphism, or an anti-$*$-isomorphism, from $L^\infty(\G_1)$ to
$L^\infty(\G_2)$.  We briefly study the intrinsic group, and prove that it coincides
with the collection of characters of $L^1(\G_2)$, as we expect from Wendel's
Theorem.  An anti-$*$-isomorphism to $L^\infty(\G_2)$ can be converted to
a $*$-isomorphism to the commutant $L^\infty(\G_2)'$ by composing with
$x\mapsto Jx^*J$; the possibility of an anti-$*$-isomorphism occurring can
of course already be seen in Walter's Theorem.  In particular, if
$L^1(\G_1)$ and $L^1(\G_2)$ are isometrically isomorphic, then $\G_1$ is
isomorphic to either $\G_2$ or $\G_2'$.  We can easily remove the
possibility of $\G_2'$ occurring by restricting to completely isometric
(or even just completely contractive)
isomorphisms between $L^1(\G_1)$ and $L^1(\G_2)$, see Section~\ref{sec:cccase}.

Having established the result for $L^1$ algebras, we can prove similar results
for quantum measure algebras-- for example, for isometric algebra isomorphisms
between the dual spaces $C_0(\G_2)^*$ and $C_0(\G_1)^*$.  Indeed, we work
with some generality, and look at C$^*$-bialgebras $(A,\Delta)$ which admit
a surjection $\pi:A\rightarrow C_0(\G)$ which intertwines the coproduct,
and such that $\pi^*$ identifies $L^1(\G)$ as an ideal in $A^*$.  This
includes the reduced and universal C$^*$-algebraic quantum groups associated
with $\G$.  As in the Kac algebra case, we use order properties of $A^{**}$
to determine $L^1(\G)$ inside $A^*$.  Our characterisation of such isometric
isomorphisms involves the intrinsic group of $A^{**}$, but we show that this
is always canonically isomorphic to the intrinsic group of $L^\infty(\G)$.
We finish to showing how, in some sense, the picture becomes clearer by
embedding everything into $L^\infty(\hat\G)$, and here the interaction
between multipliers and the antipode becomes important
(compare with \cite{daws1}).

\subsection{Acknowledgements}

The second named author would like to acknowledge the financial support of the
Marsden Fund (the Royal Society of New Zealand).

\section{Locally compact quantum groups}

We give a quick overview of the theory of locally compact quantum groups.
For readable introductions, see \cite{kusbook} or \cite{vaes}.  Our main
reference is \cite{kv}, which is a self-contained account of the C$^*$-algebraic
approach to locally compact quantum groups.  We shall however mainly work with
von Neumann algebras, for which see \cite{kusvn}.  However, this paper is not
self-contained, and should be read in conjunction with \cite{kv}.  Indeed, in
a number of places, we shall reference \cite{kv}, where really we need the
obvious von Neumann algebraic version of the required result.
See also \cite{mas} and \cite{vd}
for the C$^*$-algebraic and von Neumann algebraic approaches, respectively.

A Hopf-von Neumann algebra is a pair $(M,\Delta)$ where $M$ is a von Neumann
algebra and $\Delta:M\rightarrow M\vnten M$ is a unital norm $*$-homomorphism
which is coassociative: $(\iota\otimes\Delta)\Delta
= (\Delta\otimes\iota)\Delta$.  Then $\Delta$ induces a Banach algebra product
on the predual $M_*$.  We shall write the product in $M_*$ by juxtaposition, so
\[ \ip{x}{\omega\omega'} = \ip{\Delta(x)}{\omega\otimes\omega'}
\qquad (x\in M, \omega,\omega'\in M_*). \]
Similarly, the module actions of $M$ on $M_*$ will be denoted by
juxtaposition.

Recall the notion of a normal semi-finite faithful weight $\varphi$ on $M$
(see \cite[Chapter~VII]{tak2} for example).  We let
\[ \mf n_\varphi = \{ x\in M : \varphi(x^*x)<\infty \}, \quad
\mf m_\varphi = \lin\{ x^*y : x,y\in\mf n_\varphi \}, \quad
\mf m_\varphi^+ = \{ x\in M^+ : \varphi(x)<\infty \}. \]
Then $\mf m_\varphi$ is a hereditary $*$-subalgebra of $M$, $\mf n_\varphi$
is a left ideal, and $\mf m_\varphi^+$ is indeed $M^+ \cap \mf m_\varphi$.
We can perform the GNS construction for $\varphi$, which leads to a
Hilbert space $H$, a dense range map $\Lambda:\mf n_\varphi\rightarrow H$
and a unital normal $*$-representation $\pi:M\rightarrow\mc B(H)$ with
$\pi(x)\Lambda(y)=\Lambda(xy)$.  In future, we shall tend to suppress $\pi$.
Then $\Lambda(\mf n_\varphi \cap \mf n_\varphi^*)$ is full left Hilbert algebra,
and this contains a maximal Tomita algebra
(see \cite[Section~2, Chapter~VI]{tak2}); denote by $\mc T_\varphi\subseteq
\mf n_\varphi\cap\mf n_\varphi^*$ the inverse image under $\Lambda$ of this
maximal Tomita algebra.  Tomita-Takesaki theory gives us the modular
conjugation $J$ and the modular automorphism group $(\sigma_t)$.  Then
$\mc T_\varphi$ is a $*$-algebra, dense in $M$ for the $\sigma$-weak topology,
all of whose elements are analytic for $(\sigma_t)$.

A von Neumann algebraic quantum group is a Hopf-von Neumann algebra
$(M,\Delta)$ together with faithful normal semifinite weights $\varphi,\psi$
which are left and right invariant, respectively.  This means that
\[ \varphi\big( (\omega\otimes\iota)\Delta(x) \big)
= \varphi(x)\ip{1}{\omega}, \quad
\psi\big( (\iota\otimes\omega)\Delta(y) \big)
= \psi(y)\ip{1}{\omega}
\qquad (\omega\in M_*^+, x\in \mf m_{\varphi}^+, y\in\mf m_\psi^+). \]
Using these weights, we can construct an antipode $S$, which will in
general be unbounded.  Then $S$ has a decomposition $S = R \tau_{-i/2}$,
where $R$ is the unitary antipode, and $(\tau_t)$ is the scaling group.
The unitary antipode $R$ is a normal anti-$*$-automorphism of $M$, and
$\Delta R = \sigma(R\otimes R)\Delta$, where $\sigma:M\vnten M\rightarrow
M\vnten M$ is the tensor swap map.  As $R$ is normal, it drops to an
isometric linear map $R_*:M_*\rightarrow M_*$, which is
anti-multiplicative.  As usual, we make the canonical
choice that $\varphi = \psi\circ R$.

Let $H$ be the GNS space of $\varphi$, and let $\Lambda:\mf n_\varphi
\rightarrow H$ be the GNS map.  There is a unitary $W$, the fundamental
unitary, acting on $H\otimes H$ (the Hilbert space tensor product) such that
$\Delta(x) = W^*(1\otimes x)W$ for $x\in M$.  The left-regular representation
of $M_*$ is the map $\lambda:\omega\mapsto(\omega\otimes\iota)(W)$.  This
is a homomorphism, and the $\sigma$-weak closure of $\lambda(M_*)$ is
a von Neumann algebra $\hat M$.  We define a coproduct $\hat\Delta$ on
$\hat M$ by $\hat\Delta(x) = \hat W^*(1\otimes x)\hat W$, where $\hat W
= \Sigma W^* \Sigma$ (here $\Sigma:H\otimes H\rightarrow H\otimes H$ is the
swap map).  Then we can find invariant weights to turn $(\hat M,\hat\Lambda)$
into a locally compact quantum group-- the dual group to $M$.
We have the biduality theorem that $\hat{\hat M} = M$ canonically.

As is becoming common, we shall write $\G$ for the abstract ``object''
to be thought of as a locally compact quantum group.  We then write
$L^\infty(\G)$ for $M$, $L^1(\G)$ for $M_*$, and $L^2(\G)$ for $H$.
In this paper, we shall
often have two quantum groups $\G_1$ and $\G_2$.  Then we shall denote by
$S_i$ the antipode of $\G_i$, for $i=1,2$, and similarly for $R_i$, $\psi_i$,
and so forth.

There is of course a parallel C$^*$-algebraic theory, but we shall introduce
this below in Section~\ref{sec:cstar}.

\subsection{Isomorphisms of quantum groups}

\begin{definition}
A \emph{quantum group isomorphism} between $\G_1$ and $\G_2$ is a normal
$*$-isomorphism $\theta:L^\infty(\G_1)\rightarrow L^\infty(\G_2)$
which intertwines the coproducts.
\end{definition}

Suppose we have a $*$-isomorphism $\theta:L^\infty(\G_1)\rightarrow
L^\infty(\G_2)$ which intertwines the coproducts.
Then, arguing as in \cite[Proposition~5.45]{kv}, $\theta$ must intertwine the
antipode, the unitary antipode, and the scaling group.  As the Haar weights are
unique up to a constant, we may actually choose the weights to be intertwined
by $\theta$.  Hence every object associated to $\G_1$ is transfered to $\G_2$
by $\theta$.  

\begin{definition}
A \emph{quantum group commutant isomorphism} between $\G_1$ and $\G_2$
is a normal anti-$*$-isomorphism $\theta:L^\infty(\G_1)\rightarrow
L^\infty(\G_2)$ which intertwines the coproducts.
\end{definition}

The \emph{commutant} von Neumann algebraic quantum group
to $\G$ is $\G'$, which has $L^\infty(\G') = L^\infty(\G)'$, the commutant
of $L^\infty(\G)$, and $\Delta'(x) = (J\otimes J)\Delta(JxJ)(J\otimes J)$,
for $x\in L^\infty(\G)'$.  All the other objects (such as $W', R', \varphi'$)
associated to $\G'$ can be related to those of $\G$ using the modular conjugation
operator $J$.  See \cite[Section~4]{kusvn} for further details.  Then, if
$\theta:L^\infty(\G_1)\rightarrow L^\infty(\G_2)$ is a commutant isomorphism,
then $\theta'(x) = J\theta(x)^*J$ defines a quantum group isomorphism from
$\G_1$ to $\G_2'$; this motivates our choice of terminology.  Notice that if
$\G_2$ is commutative, then $\G_2' = \G_2$; thus we have avoided the
terminology ``quantum group anti-isomorphism'', as this would be misleading
in the motivating commutative situation.

\section{Isometries of convolution algebras}\label{sec:ell1case}

Throughout this section, fix two locally compact quantum groups $\G_1$
and $\G_2$, and let $T_*:L^1(\G_2) \rightarrow L^1(\G_1)$ be a linear bijective
isometry which is an algebra homomorphism (in short, $T_*$ is an isometric
algebra isomorphism).

Then $T=(T_*)^*:L^\infty(\G_1) \rightarrow L^\infty(\G_2)$ is a bijective linear
isometry between von Neumann algebras.  Kadison studied such maps in
\cite{kad} (see also \cite[Section~5.4]{es})
where it is shown that $T(1)$ is a unitary in $L^\infty(\G_2)$
and the map $T_1:x\mapsto T(x)T(1)^*$ is a Jordan $*$-homomorphism.  That is,
\[ T_1(x)^* = T_1(x^*), \qquad
T_1(xy+yx) = T_1(x)T_1(y) + T_1(y)T_1(x) \qquad (x,y\in L^\infty(\G_1)). \]

In our situation, we can say more about the unitary $T(1)$.

\begin{definition}
Let $\G=(M,\Delta)$ be a Hopf-von Neumann algebra.  The \emph{intrinsic
group} of $\G$ is the collection of unitaries $u\in M$ with
$\Delta(u)=u\otimes u$.
\end{definition}

Recall that a character on a Banach algebra is a non-zero multiplicative
functional.  The following is more than we need, but is of independent
interest; it generalises \cite[Theorem~3.6.10]{es} (which again makes
extensive use of a bounded antipode for a Kac algebra).  Recall that
$M(C_0(\G))$ is the multiplier algebra of $C_0(\G)$; for further details
see Section~\ref{sec:cstar} below.

\begin{theorem}\label{prop:one}
Let $\G=(M,\Delta)$ be a Hopf-von Neumann algebra.
For $x\in M$, the following are equivalent:
\begin{enumerate}
\item\label{prop:onea} $x$ is a character of the Banach algebra $M_*$;
\item\label{prop:oneb} $x\not=0$ and $\Delta(x)=x\otimes x$.
\end{enumerate}
If $\G$ is a locally compact quantum group, then a character
$x\in L^\infty(\G)$ is a unitary, and so automatically $x$ is a member
of the intrinsic group of $\G$.  Furthermore, $x\in M(C_0(\G))$ and
$x\in D(S)$ with $S(x)=x^*$.  The maps
\[ L^1(\G)\rightarrow L^1(\G); \omega\mapsto \omega x \ , \
x\omega \]
are isometric automorphisms of the algebra $L^1(\G)$.
\end{theorem}
\begin{proof}
The equivalence of (\ref{prop:onea}) and (\ref{prop:oneb}) is 
an easy calculation.

Suppose that $\G$ is a locally compact quantum group, and $x\not=0$ is
such that $\Delta(x)=x\otimes x$.  Suppose also that $x\geq0$; we shall
prove that $x=1$.  The von Neumann algebra which $x$ generates is abelian,
and so isomorphic to $L^\infty(K)$ for some measure space $K$.  Let
$\tilde x$ be the image of $x$ in $L^\infty(K)$.  We note that as
$\|x\| = \|\Delta(x)\| = \|x\otimes x\| = \|x\|^2$, necessarily $\|x\|=1$.

Let $r\in [0,1]$, and using the Borel functional calculus, let $p =
\chi_{[r,1]}(x)$.  Thus $\tilde p$ is the indicator function of the set
$\{ k\in K : \tilde x(k) \geq r \}$.  The von Neumann algebra generated by
$x\otimes x$ embeds into $L^\infty(K\times K)$ by sending $x\otimes x$ to
$\tilde x\otimes \tilde x$,
which is just the function $(k,l)\mapsto \tilde x(k)\tilde x(l)$.
Then $\chi_{[r,1]}(\tilde x\otimes\tilde x)$ is the indicator function of
the set $\{(k,l)\in K\times K : \tilde x(k)\tilde x(l)\geq r \}$.
Thus, if $\chi_{[r,1]}(\tilde x\otimes\tilde x)(k,l)=1$ then
$\tilde x(k)\tilde x(l)\geq r$ so certainly $\tilde x(k)\geq r$ (as
$\|\tilde x\|=1$) and so $(\tilde p\otimes 1)(k,l)=1$.  It follows that
\[ \chi_{[r,1]}(\tilde x\otimes\tilde x) \leq \tilde p\otimes 1. \]

By the homomorphism property of the Borel functional calculus,
\[ \Delta(p) = \chi_{[r,1]}(\Delta(x)) = \chi_{[r,1]}(x \otimes x)
\leq p \otimes 1. \]
However, we can now appeal to \cite[Lemma~6.4]{kv} to conclude that
$p=0$ or $p=1$ (as an aside on notation, $\tilde A$ as used in $\cite{kv}$
is simply $L^\infty(\G)$, see \cite[Page~874]{kv}).
So, we have that $\chi_{[r,1]}(x)=1$ or $0$ for every $r\in [0,1]$.
It follows that $x=1$.

Now let $x\in L^\infty(\G)$ be non-zero with $\Delta(x)=x\otimes x$.
As $\Delta$ is a $*$-homomorphism, it follows that $\Delta(x^*x)=
x^*x\otimes x^*x$, and so from the previous paragraph, $x^*x=1$.
Similarly, $xx^*=1$, so $x$ is a unitary, as required.

Then
\[ 1\otimes x = (x^*\otimes 1)\Delta(x), \qquad
1\otimes x^* = \Delta(x^*)(x\otimes 1), \]
and so from (the von Neumann algebraic analogue of)
\cite[Proposition~5.33]{kv} we conclude that $x\in D(S)$ with $S(x)=x^*$.
To show that $x$ is a multiplier of $C_0(\G)$, we adapt an idea from
\cite[Section~4]{woro}, which in turn is inspired by \cite[Page~441]{bs}.
We have that $W\in M(C_0(\G) \otimes \mc K(L^2(\G)))$, where
$\mc K(L^2(\G))$ is the compact operators on $L^2(\G)$, see
\cite[Section~3.4]{kv} or compare \cite[Theorem~1.5]{woro}.  Then
\[ x\otimes 1 = (1\otimes x^*) \Delta(x) =
(1\otimes x^*)W^*(1\otimes x)W \in M(C_0(\G) \otimes \mc K(L^2(\G))), \]
and so $x\in M(C_0(\G))$ as required.

Finally, for $\omega,\omega'\in L^1(\G)$, we see that
\[ \ip{y}{(\omega\omega')x}
= \ip{(x\otimes x)\Delta(y)}{\omega\otimes\omega'}
= \ip{y}{(\omega x)(\omega' x)} \qquad (y\in L^\infty(\G)), \]
so the map $\omega\mapsto\omega x$ is an algebra homomorphism, with
inverse $\omega\mapsto\omega x^*$.  The case of $\omega\mapsto x\omega$
is analogous.
\end{proof}

We remark that similar results to the above theorem have been obtained
independently by Neufang and Kalantar, see Kalantar's thesis,
\cite[Theorem~3.2.11]{kt} and \cite[Theorem~3.9]{kn}.

We hence see that if $T_*:L^1(\G_2)\rightarrow L^1(\G_1)$ is an
isometric algebra isomorphism, then so is $T_{1,*}:\omega\mapsto
T_*(T(1)^*\omega)$.  For the rest of this section, we shall just
assume that actually $T(1)=1$.

Let $p\in L^\infty(\G_2)$ be a central projection, and let $T_p$ be the
map $x\mapsto T(x)p$.  As in \cite[Section~5.4]{es}, we define
\begin{align*} \mc P_h &= \big\{ p\text{ a central projection in }
L^\infty(\G_2) \text{ with }T_p\text{ an algebra homomorphism}\big\}, \\
\mc P_a &= \big\{ p\text{ a central projection in }L^\infty(\G_2)
\text{ with }T_p\text{ an algebra anti-homomorphism	}\big\}. \end{align*}
Then \cite[Lemma~5.4.5]{es} shows that both $\mc P_a$ and $\mc P_h$
have greatest elements, say $s_a$ and $s_h$.  From \cite[Theorem~3.3]{sto},
there is some $p\in\mc P_a$ with $1-p\in\mc P_h$, and so $s_a + s_h\geq 1$.

The following results are also shown in \cite{es}, but we give
sketch proofs to verify that the results still hold for locally compact
quantum groups.

\begin{lemma}
Let $x\in L^\infty(\G)$ be a central projection with $\Delta(x)\geq x\otimes x$
and $R(x)=x$.  Then $W(x\otimes x) = (x\otimes x)W$ and $\Delta(x)(x\otimes 1)
= \Delta(x)(1\otimes x) = x\otimes x$.
\end{lemma}
\begin{proof}
We have that $x\otimes x = (x\otimes x)\Delta(x) = (x\otimes x)W^*(1\otimes x)W$,
and so $(x\otimes x)W^*(1\otimes x) = (x\otimes x)W^*$.
Now we use that $(\hat J\otimes J)W (\hat J\otimes J) = W^*$, see
\cite[Corollary~2.2]{kusvn}.  Thus
\[ x\otimes x = (x\otimes x)(\hat J\otimes J)W(\hat J\otimes J)(1\otimes x)
(\hat J\otimes J)W^*(\hat J\otimes J), \]
but $JxJ = x^*=x$ as $x$ is central and self-adjoint, and $\hat J x \hat J
=R(x^*)=R(x)=x$ by assumption.  So $x\otimes x = (x\otimes x)W(1\otimes x)W^*$.
Taking adjoints gives $x\otimes x = W(1\otimes x)W^*(x\otimes x)$.
As $W^*\in L^\infty(\G) \vnten L^\infty(\hat\G)$, we see that
$W^*(x\otimes x) = (x\otimes 1)W^*(1\otimes x)$, and so, from above,
\[ x\otimes x = W(x\otimes x)W^*(1\otimes x) = W(x\otimes x)W^*. \]
Thus $W(x\otimes x) = (x\otimes x)W$.

Then, arguing similarly, $\Delta(x)(x\otimes 1) = W^*(1\otimes x)W(x\otimes 1)
= W^*(x\otimes x)W = x\otimes x$.  The case of $\Delta(x)(1\otimes x)$
follows by applying the result to $\G^\op$ (see \cite[Section~4]{kusvn}).
\end{proof}

\begin{corollary}
Let $p,q\in L^\infty(\G)$ be central projections with
$\Delta(p)\geq p\otimes p$ and $\Delta(q)\geq q\otimes q$,
with $R(p)=p$ and $R(q)=q$, and with $p+q\geq 1$.  Then $p=1$ or $q=1$.
\end{corollary}
\begin{proof}
By the lemma, $\Delta(p)( (1-p)\otimes p) = \Delta(p)(1\otimes p)
- p\otimes p = 0$, and $\Delta(q)(q\otimes(1-q))=0$.  As $1-q\leq p$
and $1-p\leq q$, it follows that $\Delta(p)( (1-p)\otimes(1-q) ) = 0$
and $\Delta(q)((1-p)\otimes(1-q))=0$.  As $\Delta(p)+\Delta(q)\geq 1$,
it follows that $(1-p)\otimes(1-q)=0$, so $p=1$ or $q=1$.
\end{proof}

\begin{proposition}\label{prop:two}
Form $S_a$ and $S_h$ as above.  Then:
\begin{enumerate}
\item $(T_{S_h} \otimes T_{S_h})\Delta_1(x) = \Delta_2(T(x))(S_h\otimes S_h)$
for $x\in L^\infty(\G_1)$;
\item\label{prop:twoa}
$(T_{S_a} \otimes T_{S_a})\Delta_1(x) = \Delta_2(T(x))(S_a\otimes S_a)$
for $x\in L^\infty(\G_1)$;
\item $\Delta_2(S_h) \geq S_h \otimes S_h$;
\item $\Delta_2(S_a) \geq S_a \otimes S_a$.
\end{enumerate}
\end{proposition}
\begin{proof}
We prove claims for $S_a$; the proofs for $S_h$ are easier.
The preadjoint of $T_{S_a}$ is the map $\omega\mapsto T_*(S_a\omega)$.
Firstly, let $\omega,\omega'\in L^1(\G_2)$, and calculate
\[ \ip{(T_{S_a} \otimes T_{S_a})\Delta_1(x)}{\omega\otimes\omega'}
= \ip{x}{T_*(S_a\omega)T_*(S_a\omega')}
= \ip{\Delta_2(T(x))}{S_a\omega\otimes S_a\omega'}, \]
which shows (\ref{prop:twoa}).

As $S_a$ is central, we see that $S_a\otimes S_a \in L^\infty(\G_2)'
\vnten L^\infty(\G_2)' \subseteq \Delta_2(L^\infty(\G_2))'$.  Let $q\in
L^\infty(\G_2)$ be such that $\Delta_2(q)$ is the central support of
$S_a\otimes S_a$ (so $q$ is the smallest central projection with $\Delta_2(q)
(S_a\otimes S_a) = S_a\otimes S_a$).  Then
\[ \Phi: \Delta_2(L^\infty(\G_2))(S_a\otimes S_a) \rightarrow
\Delta_2(L^\infty(\G_2)q); \quad \Delta_2(x)(S_a\otimes S_a)
\mapsto \Delta_2(xq) = \Delta_2(x)\Delta_2(q), \]
is readily seen to be an isomorphism.  Then, for $x\in L^\infty(\G_1)$,
\[ \Delta_2(T_q(x)) = \Delta_2(T(x)q)
= \Phi\big( \Delta_2(T(x)) (S_a\otimes S_a) \big)
= \Phi\big( (T_{S_a}\otimes T_{S_a})\Delta_1(x) \big). \]
So $x\mapsto \Delta_2(T_q(x))$ is anti-multiplicative, and so $q\in \mc P_a$.
Thus $q\leq S_a$, and so $\Delta_2(S_a) \geq \Delta(q) \geq S_a\otimes S_a$
as required.
\end{proof}

At this point, we can no longer follow \cite{es}.  We would like to
show that $TR_1 = R_2T$ (that is, $T'$ as defined in the next proposition,
is the identity map) but we have to proceed somewhat indirectly.

\begin{proposition}\label{prop:five}Suppose that the map
$T' = T^{-1} R_2 T R_1:L^\infty(\G_1) \rightarrow
L^\infty(\G_1)$ is a homomorphism.  Then $T$ is either a $*$-homomorphism or
an anti-$*$-homomorphism.
\end{proposition}
\begin{proof}
As the unitary antipode $R_2$ is an anti-$*$-homomorphism, it is easy to
see that $R_2(S_h)$ is a central projection.  For $x\in L^\infty(\G_1)$,
\[ T_{R_2(S_h)}(x) = T(x) R_2(S_h) = R_2\big( R_2(T(x)) S_h \big)
= R_2\big( T (T' (R_1(x))) S_h \big). \]
As $y\mapsto T(y)S_h$ is a homomorphism, it follows that $T_{R_2(S_h)}$
is a homomorphism, and so $R_2(S_h) \leq S_h$.  As $R_2$ preserves the
order, also $S_h \leq R_2(S_h)$, so $S_h = R_2(S_h)$.

A similar argument establishes that $R(S_a) = S_a$.  So,
combining the previous proposition and corollary, we conclude that
either $S_h=1$, in which case $T$ is a $*$-homomorphism, or $S_h=0$, so
$S_a=1$, and $T$ is an anti-$*$-homomorphism.
\end{proof}

We are henceforth motivated to study the map $T' = T^{-1} R_2 T R_1$.
Notice that this map is normal, and the preadjoint $T'_*$ is
an isometric algebra isomorphism from $L^1(\G_2)$ to itself.

\subsection{Characterising the unitary antipode}\label{sec:ua}

We now study the unitary antipode more closely.
For us, an important characterisation
of $R$ is the following, given in \cite[Proposition~5.20]{kv}:
\[ R\big( (\psi\otimes\iota)((a^*\otimes 1)\Delta(b)) \big)
= (\psi\otimes\iota)\big( \Delta(\sigma^\psi_{-i/2}(a^*))
(\sigma^\psi_{-i/2}(b)\otimes 1) \big), \]
where $a,b\in\mc T_\psi$.  (We shall shortly explain further exactly what
this formula means).
We are hence motivated to look at the right Haar weights, and how they
interact with $T$.  We shall then split $L^\infty(\G_1)$ into a direct
summand, with $T$ acting as a homomorphism in the first component, and
as an anti-homomorphism in the second.  Then $R_1$ and $R_2$ will interact
well with $T$ on these components, but less well on the cross-terms.  However,
this ``bad interaction'' will cancel out if we consider $T'^2$,
for $T'$ as defined above.

\begin{lemma}
The map $L^\infty(\G_1)^+ \rightarrow [0,\infty]; x\mapsto \psi_2(T(x))$
is a right-invariant, normal semi-finite faithful weight on $L^\infty(\G_1)$,
and is hence proportional to $\psi_1$.
\end{lemma}
\begin{proof}
As $T$ is a Jordan homomorphism, it restricts to an order isomorphism
$L^\infty(\G_1)^+ \rightarrow L^\infty(\G_2)^+$.  Thus we can define
$\psi = \psi_2\circ T:L^\infty(\G_1)^+\rightarrow[0,\infty]$, and it follows
that $\psi$ is a faithful weight, and $\mf m_\psi^+ = T^{-1}(\mf m_{\psi_2}^+)$.
Thus also $\mf m_\psi = T^{-1}(\mf m_{\psi_2})$.  As $T$ is $\sigma$-weakly
continuous, it is now routine to establish that $\psi$ is semi-finite,
and normal (as $T$ is an order isomorphism on the positive cones).

It remains to check that $\psi$ is right-invariant.  For $\omega\in
L^1(\G_1)^+$ and $y\in\mf m_{\psi}^+$, a simple calculation shows that
$T((\iota\otimes\omega)\Delta_1(y))
= (\iota\otimes T_*^{-1}(\omega))\Delta_2(T(y))$.
As $T_*^{-1}(\omega)\geq 0$ and $T(y) \in \mf m_{\psi_2}^+$, it follows that
\[ \psi\big( (\iota\otimes\omega)\Delta(y) \big)
= \psi_2\big( (\iota\otimes T_*^{-1}(\omega))\Delta_2(T(y)) \big)
= \ip{1}{T_*^{-1}(\omega)} \psi_2(T(y)). \]
As $T$ is unital, this shows that $\psi$ is right-invariant.
\end{proof}

Henceforth, we shall actually assume that that $\psi_1 = \psi_2\circ T$.

Henceforth, using \cite[Theorem~3.3]{sto}, we fix a central projection
$p\in L^\infty(\G_2)$ such that $T_p$ is a homomorphism, and $T_{1-p}$
is an anti-homomorphism.  Note that we cannot
necessarily assume that $p=S_a$ and $1-p=S_h$.  Let $q = T^{-1}(p)$.

\begin{lemma}
With $p,q$ as above, we have that $q$ is a central projection in 
$L^\infty(\G_1)$.  Then $L^\infty(\G_1)$ decomposes as
$q L^\infty(\G_1) \oplus (1-q) L^\infty(\G_1)$, $L^\infty(\G_2)$
decomposes as $p L^\infty(\G_2) \oplus (1-p) L^\infty(\G_2)$, and
under these identifications, $T$ decomposes as $T_p \oplus T_{1-p}$.
\end{lemma}
\begin{proof}
Let $x\in L^\infty(\G_1)$.  Then $T(xq)p =  T_p(xq) = T_p(x) T_p(q)
= T(x)p T(q)p = T(x)p = T_p(x)$, and similarly $T_p(qx) = T_p(x)$, and
$T_{1-p}(qx)=T_{1-p}(xq)=0$.  Thus
\[ T(xq-qx) = T_p(xq-qx) + T_{1-p}(xq-qx) = T_p(xq) - T_p(qx)
= T_p(x) - T_p(x)=0. \]
So $q$ is central; it is easily seen to be a projection.  The remaining
claims now follow by simple calculation.
\end{proof}

This lemma means that, for example, given $a\in qL^\infty(\G_1)$ and
$x\in L^\infty(\G_1)$,
\[ T(ax) = T\big( axq + ax(1-q) \big) = T_p(a) T_p(x)
= T_p(a) T(x) = T(a) T_p(x) = T(a) T(x). \]
Thus we understand $T$ quite well; what is unclear is how $T$ interacts
with the unitary antipodes $R_1$ and $R_2$.

We can then restrict $\psi_1 = \psi_2\circ T$ to $qL^\infty(\G_1)$ and
to $(1-q)L^\infty(\G_2)$, say giving $\psi^q_1$ and $\psi^{1-q}_1$.
As $T_p$ is a $*$-homomorphism, it is clear that $T_p$ gives a bijection
from $\mf n_{\psi_1^q}$ to $\mf n_{\psi_2^p}$.  As $T_{1-p}$ is an
anti-$*$-homomorphism, we have that $x\in \mf n_{\psi_1^{1-q}}$ if
and only if $T(x^*)=T(x)^*\in\mf n_{\psi_2^{1-p}}$.
To ease notation for the modular automorphism groups, for $t\in\mathbb R$,
we shall let $\sigma^{2,p}_t = \sigma^{\psi_2^p}_t$
and $\sigma^{2,1-p}_t = \sigma^{\psi_2^{1-p}}_t$, and similarly for $\psi_1$.

\begin{lemma}\label{lem:three}
The map $T$ intertwines the modular automorphism groups in the following ways:
\[ T_p \circ \sigma^{1,q}_t = \sigma^{2,p}_t \circ T_p, \qquad
T_{1-p} \circ \sigma^{1,1-q}_t
= \sigma^{2,1-p}_{-t} \circ T_{1-p}
\qquad (t\in\mathbb R). \]
\end{lemma}
\begin{proof}
As $T_p$ is a $*$-isomorphism between $L^\infty(\G_1)q$ and $L^\infty(\G_2)p$,
it is standard that it intertwines the modular automorphism group, compare
\cite[Corollary~1.4, Chapter~VIII]{tak2}.  As $T_{1-p}$ is an
anti-$*$-isomorphism, a variant of the standard argument will show that
we get the sign change $t\mapsto -t$.
\end{proof}

As in \cite[Section~6]{kusbook} (see also the C$^*$-algebraic
version in \cite[Section~1.5]{kv}) we let 
\[ \mf m_{\psi_1\otimes\iota}^+ = \big\{ x\in (L^\infty(\G_1)
\vnten L^\infty(\G_1))^+ : (\iota\otimes\omega)(x) \in \mf m_{\psi_1}^+
\ (\omega\in L^1(\G_1)_+ \big\}. \]
Then $\mf m_{\psi_1\otimes\iota}^+$ is a hereditary cone in
$(L^\infty(\G_1)\vnten L^\infty(\G_1))^+$.  Let $\mf m_{\psi_1\otimes\iota}$
be the $*$-subalgebra generated by $\mf m_{\psi_1\otimes\iota}^+$; this
agrees with the linear span of $\mf m_{\psi_1\otimes\iota}^+$.  There is
a linear map
\[ (\psi_1\otimes\iota) : \mf m_{\psi_1\otimes\iota}
\rightarrow L^\infty(\G_1) \quad\text{with}\quad
\ip{ (\psi_1\otimes\iota)(x) }{\omega}
= \psi_1\big( (\iota\otimes\omega)x \big). \]
We then set
\[ \mf n_{\psi_1\otimes\iota} = \big\{ x\in L^\infty(\G_1)
\vnten L^\infty(\G_1) : x^*x \in \mf m_{\psi_1\otimes\iota}^+ \big\}. \]
This is a left ideal in $L^\infty(\G_1)\vnten L^\infty(\G_1)$, and
$\mf m_{\psi_1\otimes\iota}$ is the linear span of
$\mf n_{\psi_1\otimes\iota}^* \mf n_{\psi_1\otimes\iota}$.

As $\psi_1$ is right-invariant, a simple calculation shows that for
$a,b\in\mf n_{\psi_1}$, we have that $\Delta(b)\in\mf n_{\psi_1\otimes\iota}$
and that $a\otimes 1 \in \mf n_{\psi_1\otimes\iota}$.  Thus
$(a^*\otimes 1)\Delta(b) \in \mf m_{\psi_1\otimes\iota}$, and
similarly $\Delta(a^*)(b\otimes 1) \in \mf m_{\psi_1\otimes\iota}$.

In particular, for $a,b\in\mc T_{\psi_1}$, we can make sense of the formula
\[ R_1\big( (\psi_1\otimes\iota)((a\otimes 1)\Delta_1(b)) \big)
= (\psi_1\otimes\iota)\big( \Delta_1(\sigma^{\psi_1}_{-i/2}(a))
(\sigma^{\psi_1}_{-i/2}(b)\otimes 1) \big). \]

\begin{lemma}\label{lem:one}
Let $a\in\mc T_{\psi_1}$ and $b\in\mc T_{\psi_1}q$.  Then
\begin{gather*} T\big( (\psi_1\otimes\iota)((b\otimes 1)\Delta_1(a)) \big)
= (\psi_2\otimes\iota)\big(  (T(b)\otimes 1)\Delta_2(T(a)) \big), \\
T\big( (\psi_1\otimes\iota)(\Delta_1(a)(b\otimes 1)) \big)
= (\psi_2\otimes\iota)\big(  \Delta_2(T(a))(T(b)\otimes 1) \big)
\end{gather*}
\end{lemma}
\begin{proof}
Let $\omega,\omega'\in L^1(\G_2)$.  Then, for $x\in L^\infty(\G_1)$,
\[ \ip{x}{T_*(\omega')b} = \ip{T(bx)}{\omega'} = \ip{T(b)T(x)}{\omega'}
= \ip{x}{T_*(\omega' T(b))}, \]
using that $b\in L^\infty(\G_1)q$.  Thus
\begin{align*}
& \ip{T\big( (\iota\otimes T_*(\omega))((b\otimes 1)\Delta_1(a)) \big)}{\omega'}
= \ip{(b\otimes 1)\Delta_1(a)}{T_*(\omega') \otimes T_*(\omega)} \\
&\qquad= \ip{\Delta_1(a)}{T_*(\omega' T(b)) \otimes T_*(\omega)} 
= \ip{T(a)}{(\omega' T(b))\omega}
= \ip{(T(b)\otimes 1)\Delta_2(T(a))}{\omega'\otimes\omega}. \end{align*}
Hence finally
\begin{align*}
& \ip{T\big( (\psi_1\otimes\iota)((b\otimes 1)\Delta_1(a)) \big)}{\omega}
= \psi_1\big( (\iota\otimes T_*(\omega))((b\otimes 1)\Delta_1(a)) \big) \\
&\qquad= \psi_2\big( (\iota\otimes\omega)(T(b)\otimes 1)\Delta_2(T(a)) \big)
= \ip{(\psi_2\otimes\iota)\big(  (T(b)\otimes 1)\Delta_2(T(a)) \big)}{\omega},
\end{align*}
as required.

Now, as $\psi_2$ is a weight, we have that $\psi_2(x^*)=\overline{\psi_2(x)}$ for
$x\in\mf m_{\psi_2}$.  We can also verify that $(\iota\otimes\omega)(x^*)
= (\iota\otimes\omega^*)(x)^*$ for $x\in\mf m_{\psi_2\otimes\iota}$ and
$\omega\in L^1(\G_2)^+$.  It follows that $(\psi_2\otimes\iota)(x^*)
= (\psi_2\otimes\iota)(x)^*$.  As $\mc T_{\psi_1}$ is a $*$-algebra,
and $T$ respects the involution, applying this calculation to $a^*$ and $b^*$
and then taking the adjoint yields the second claimed equality.
\end{proof}

\begin{lemma}\label{lem:two}
Let $a\in\mc T_{\psi_1}$ and $b\in\mc T_{\psi_1}(1-q)$.  Then
\begin{gather*} T\big( (\psi_1\otimes\iota)((b\otimes 1)\Delta_1(a)) \big)
= (\psi_2\otimes\iota)\big(  \Delta_2(T(a))(T(b)\otimes 1) \big), \\
T\big( (\psi_1\otimes\iota)(\Delta_1(a)(b\otimes 1)) \big)
= (\psi_2\otimes\iota)\big(  (T(b)\otimes 1) \Delta_2(T(a))  \big).
\end{gather*}
\end{lemma}
\begin{proof}
As in the previous proof, but now using that $b\in L^\infty(\G_1)(1-q)$,
we check that for $\omega,\omega'\in L^1(\G_2)$, we have that
$T_*(\omega')b = T_*\big( T(b) \omega' \big)$, which leads to
\[ T\big( (\iota\otimes T_*(\omega))((b\otimes 1)\Delta_1(a)) \big)
= (\iota\otimes\omega) \big( \Delta_2(T(a)) (T(b)\otimes 1) \big), \]
which gives the first result.  The second equality now follows by
taking adjoints.
\end{proof}

\begin{proposition}\label{prop:three}
As before, let $T' = T^{-1} R_2 T R_1$.  If $a,b\in\mc T_{\psi_1}q$
or $a,b\in\mc T_{\psi_1}(1-q)$, we have that
\[ T'\big( (\psi_1\otimes\iota)((b\otimes 1)\Delta_1(a)) \big)
= (\psi_1\otimes\iota)((b\otimes 1)\Delta_1(a)). \]
\end{proposition}
\begin{proof}
Suppose that $a,b\in \mc T_{\psi_1}(1-q)$, the other case being analogous.
We have that
\begin{align*}
& T R_1 \big( (\psi_1\otimes\iota)((b\otimes 1)\Delta_1(a)) \big)
= T\big( (\psi_1\otimes\iota)\big( \Delta_1(\sigma^{1,1-q}_{-i/2}(b))
(\sigma^{1,1-q}_{-i/2}(a)\otimes 1) \big) \big) \\
&\qquad = (\psi_2\otimes\iota)\big( (T_{1-p}\sigma^{1,1-q}_{-i/2}(a)\otimes 1)
\Delta_2(T_{1-p}\sigma^{1,1-q}_{-i/2}(b)) \big) \\
&\qquad= (\psi_2\otimes\iota)\big( (\sigma^{2,1-p}_{i/2}T_{1-p}(a)\otimes 1)
\Delta_2(\sigma^{2,1-p}_{i/2}T_{1-p}(b)) \big)
\end{align*}
using first Lemma~\ref{lem:two} (applied to $\sigma^{1,1-q}_{-i/2}(a)
\in \mc T_{\psi_1}(1-q)$) and then Lemma~\ref{lem:three}.

Thus, taking adjoints gives that
\begin{align*}
& T R_1 \big( (\psi_1\otimes\iota)((b\otimes 1)\Delta_1(a)) \big)
= (\psi_2\otimes\iota)\big( \Delta_2(\sigma^{2,1-p}_{-i/2}T_{1-p}(b^*))
(\sigma^{2,1-p}_{-i/2}T_{1-p}(a^*)\otimes 1) \big) ^* \\
&\qquad= R_2\big( (\psi_2\otimes\iota)\big( (T_{1-p}(b^*)\otimes 1)
\Delta_2(T_{1-p}(a^*)) \big) ^*\big) \\
&\qquad= R_2 T \big( (\psi_1\otimes\iota)\big( \Delta_1(a^*)(b^*\otimes 1)
   \big)^* \big)
= R_2 T \big( (\psi_1\otimes\iota) \big((b\otimes 1)\Delta_1(a)\big) \big),
\end{align*}
as required.
\end{proof}

\begin{proposition}\label{prop:four}
As before, let $T' = T^{-1} R_2 T R_1$.  If $a\in\mc T_{\psi_1}(1-q)$ and
$b\in \mc T_{\psi_1}q$, or vice versa, we have that
\begin{gather*} T'\big( (\psi_1\otimes\iota)((b\otimes 1)\Delta_1(a)) \big)
= (\psi_1\otimes\iota)(\Delta_1(a)(\sigma^{\psi_1}_{-i}(b)\otimes 1)), \\
T'\big( (\psi_1\otimes\iota)(\Delta_1(a)(b\otimes 1)) \big)
= (\psi_1\otimes\iota)((\sigma^{\psi_1}_{i}(b)\otimes 1)\Delta_1(a)).
\end{gather*}
\end{proposition}
\begin{proof}
Suppose that $a\in\mc T_{\psi_1}(1-q)$ and $b\in \mc T_{\psi_1}q$, so we can
follow the previous proof through to get that
\[ T R_1 \big( (\psi_1\otimes\iota)((b\otimes 1)\Delta_1(a)) \big)
= (\psi_2\otimes\iota)\big( (\sigma^{2,1-p}_{i/2}T_{1-p}(a)\otimes 1)
\Delta_2(\sigma^{2,p}_{-i/2}T_{p}(b)) \big), \]
where here we remember that $b\in \mc T_{\psi_1}q$.  Thus
\begin{align*}
& T R_1 \big( (\psi_1\otimes\iota)((b\otimes 1)\Delta_1(a)) \big)
= (\psi_2\otimes\iota)\big( \Delta_2(\sigma^{2,p}_{i/2}T_{p}(b^*))
(\sigma^{2,1-p}_{-i/2}T_{1-p}(a^*)\otimes\iota) \big) ^* \\
&\qquad= R_2\big( (\psi_2\otimes\iota)\big( (\sigma_i^{2,p}(T_{p}(b^*))\otimes 1)
\Delta_2(T_{1-p}(a^*)) \big) ^*\big) \\
&\qquad= R_2 T \big( (\psi_1\otimes\iota)\big( (\sigma_i^{1,q}(b^*)\otimes 1)
   \Delta_1(a^*) \big)^* \big)
= R_2 T \big( (\psi_1\otimes\iota) \big(\Delta_1(a)
   (\sigma^{\psi_1}_{-i}(b)\otimes 1) \big) \big),
\end{align*}
as required, using Lemma~\ref{lem:one}.  The case when $a\in\mc T_{\psi_1}q$
and $b\in\mc T_{\psi_1}(1-q)$ follows similarly.  Again, taking adjoints
(and remembering that $\sigma^{\psi_1}_i(b)^* = \sigma^{\psi_1}_{-i}(b^*)$
gives the second claimed equality).
\end{proof}

\begin{corollary}\label{corr:one}
We have that $T'^2 = \iota$.
\end{corollary}
\begin{proof}
By density, it is enough to verify these identities on
elements of the form $(\psi_1\otimes\iota)((b\otimes 1)\Delta_1(a))$
for $a,b\in\mc T_{\psi_1}$.  By linearity, we may suppose that
$a,b\in\mc T_{\psi_1}q$ or $a,b\in\mc T_{\psi_1}(1-q)$,
in which case the result follows from Proposition~\ref{prop:three},
or that $a\in\mc T_{\psi_1}q, b\in\mc T_{\psi_1}(1-q)$ or vice versa,
in which case the result follows from Proposition~\ref{prop:four}.
\end{proof}

Finally, we wish to show that $T'$ commutes with the scaling group $(\tau_t)$.
For this, recall from \cite[Proposition~6.8]{kv} that
$\Delta_1 \sigma^{\psi_1}_t = (\sigma^{\psi_1}_t \otimes \tau_{-t})\Delta_1$.

\begin{proposition}\label{prop:six}
We have that $\tau_t T' = T' \tau_t$ for each $t\in\mathbb R$.
\end{proposition}
\begin{proof}
Let $a\in\mc T_{\psi_1}(1-q)$ and $b\in\mc T_{\psi_1}q$.
Then, from Proposition~\ref{prop:four},
\begin{align*}
\tau_t T'\big( (\psi_1\otimes\iota)((b\otimes 1)\Delta_1(a)) \big)
= \tau_t(\psi_1\otimes\iota)(\Delta_1(a)(\sigma^{\psi_1}_{-i}(b)\otimes 1)).
\end{align*}
Now, for $\omega,\omega'\in L^1(\G_1)$,
\begin{align*}
&\ip{ (\iota\otimes \omega\circ\tau_t)(\Delta_1(a)
   (\sigma^{\psi_1}_{-i}(b)\otimes 1)) }{\omega'}
= \ip{(\iota\otimes\tau_t)\Delta_1(a)}
   {\sigma^{\psi_1}_{-i}(b)\omega'\otimes\omega} \\
&\qquad = \ip{(\sigma^{\psi_1}_t\otimes\iota)\Delta_1(\sigma^{\psi_1}_{-t}(a))}
   {\sigma^{\psi_1}_{-i}(b)\omega'\otimes\omega}
= \ip{(\sigma^{\psi_1}_t\otimes\iota)\big(\Delta_1(\sigma^{\psi_1}_{-t}(a))
   (\sigma^{\psi_1}_{-t-i}(b)\otimes 1)\big)}
   {\omega'\otimes\omega} \\
&\qquad = \ip{ \sigma^{\psi_1}_t\big( (\iota\otimes\omega)
   \big(\Delta_1(\sigma^{\psi_1}_{-t}(a))(\sigma^{\psi_1}_{-t-i}(b)\otimes 1)\big)
   \big) }{\omega'}
\end{align*}
Thus also
\begin{align*} &\ip{ \tau_t(\psi_1\otimes\iota)((b\otimes 1)\Delta_1(a)) }{\omega}
= \psi_1\big( (\iota\otimes \omega\circ\tau_t)(\Delta_1(a)
   (\sigma^{\psi_1}_{-i}(b)\otimes 1)) \big) \\
&\qquad = \psi_1\big( \sigma^{\psi_1}_t\big( (\iota\otimes\omega)
   \big(\Delta_1(\sigma^{\psi_1}_{-t}(a))(\sigma^{\psi_1}_{-t-i}(b)\otimes 1)\big)
   \big) \big) \\
&\qquad = \ip{(\psi_1\otimes\iota)
   \big(\Delta_1(\sigma^{\psi_1}_{-t}(a))(\sigma^{\psi_1}_{-t-i}(b)\otimes 1)\big)}
   {\omega},
\end{align*}
and so we conclude that
\begin{align*} &\tau_t T'\big( (\psi_1\otimes\iota)((b\otimes 1)\Delta_1(a)) \big)
= (\psi_1\otimes\iota)
   \big(\Delta_1(\sigma^{\psi_1}_{-t}(a))(\sigma^{\psi_1}_{-t-i}(b)\otimes 1)\big).
\end{align*}

Similarly, we find that
\begin{align*}
&\ip{ (\iota\otimes \omega\circ\tau_t)((b\otimes 1)\Delta_1(a))}{\omega'}
= \ip{ (\iota\otimes\tau_t)\Delta_1(a) }{ \omega'b \otimes \omega} \\
&\qquad = \ip{(\sigma^{\psi_1}_t\otimes\iota)\Delta_1(\sigma^{\psi_1}_{-t}(a))}
   {\omega'b \otimes \omega}
= \ip{(\sigma^{\psi_1}_t\otimes\iota)
   \big((\sigma^{\psi_1}_{-t}(b)\otimes 1)\Delta_1(\sigma^{\psi_1}_{-t}(a))\big)}
   {\omega' \otimes \omega}.
\end{align*}
So arguing similarly,
\begin{align*}
&T' \tau_t \big( (\psi_1\otimes\iota)((b\otimes 1)\Delta_1(a)) \big)
= T' \big( (\psi_1\otimes\iota)((\sigma^{\psi_1}_{-t}(b)\otimes 1)
   \Delta_1(\sigma^{\psi_1}_{-t}(a))) \big) \\
&\qquad = (\psi_1\otimes\iota)
   \big(\Delta_1(\sigma^{\psi_1}_{-t}(a))(\sigma^{\psi_1}_{-t-i}(b)\otimes 1)\big)
=\tau_t T'\big( (\psi_1\otimes\iota)((b\otimes 1)\Delta_1(a)) \big).
\end{align*}

The same argument works if $a\in\mc T_{\psi_1}q$ and $b\in\mc T_{\psi_1}(1-q)$.
Similarly, by using Proposition~\ref{prop:three}, a similar calculation
works for $a,b\in \mc T_{\psi_1}q$ or $a,b\in \mc T_{\psi_1}(1-q)$.
By linearity and density, the result follows.
\end{proof}

\subsection{The main result}

We are now in a position to state and prove our main result.

\begin{theorem}\label{thm:main}
Let $T_*:L^1(\G_2) \rightarrow L^1(\G_1)$ be an isometric algebra
isomorphism.  Then $u = T(1)\in L^\infty(\G_2)$ is a member of the intrinsic
group, and there is a quantum group isomorphism, or quantum group commutant
isomorphism, $\theta:L^\infty(\G_1) \rightarrow L^\infty(\G_2)$ such that
\[ T_*(\omega) = \theta_*(u\omega) \qquad (\omega\in L^\infty(\G_2)). \]
In particular, $\G_1$ is isomorphic to either $\G_2$ or $\G_2'$.
\end{theorem}
\begin{proof}
Suppose that the result holds (with $u=1$) when $T=(T_*)^*$ is unital.
Then we apply this to $T_1$ to find that
\[ T_*(T(1)^*\omega) = T_{1,*}(\omega) = \theta_*(\omega)
\qquad (\omega\in L^1(\G_2)), \]
from which the general case follows.

So, we may suppose that $T$ is unital.  We wish to prove that
$T$ is either a $*$-homomorphism, or an anti-$*$-homomorphism.
Form $T' = T^{-1}R_2TR_1$.
By Corollary~\ref{corr:one}, $T'^2=\iota$, so $R_1 T' R_1 = R_1 T^{-1}R_2T
= T'^{-1} = T'$; thus $T'$ commutes with $R_1$.

Now, $T'(1)=1$ and $T'_*$ is an isometric algebra isomorphism.
By Proposition~\ref{prop:five}, as $T'^{-1} R_1 T' R_1 = \iota$, it
follows that $T'$ is either a $*$-homomorphism, or an anti-$*$-homomorphism.
If $T'$ is a $*$-homomorphism, then Proposition~\ref{prop:five} now
shows that $T$ itself is either a $*$-homomorphism or an anti-$*$-homomorphism,
as required.

If we reverse the roles of $\G_1$ and $\G_2$, and work with $T^{-1}$,
then the same arguments show that $(T^{-1})' = T R_1 T^{-1} R_2$ is
either a $*$-homomorphism or an anti-$*$-homomorphism.  If it is a
$*$-homomorphism, then $T^{-1}$ (and so $T$) is either a $*$-homomorphism
or an anti-$*$-homomorphism, as required.

So, the remaining case is when both $T'$ and $(T^{-1})'$ are
anti-$*$-homomorphisms (and, to avoid special cases, by this we mean that
$T'$ and $(T^{-1})'$ are not also $*$-homomorphisms).
Then we can consider the map $\Phi:L^\infty(\G_1)
\rightarrow L^\infty(\G_1') = L^\infty(\G_1)'; x\mapsto JT'(x)^*J$, which
will be a $*$-isomorphism which intertwines the coproducts.  Thus $\Phi$ will
also intertwine the antipode, the unitary antipode, and in particular
the scaling group.  The scaling group of $L^\infty(\G_1')$ is
$\tau'_t(x) = J\tau_{-t}(JxJ)J$, see \cite[Section~4]{kusvn}.  So, for
$x\in L^\infty(\G_1)$,
\[ JT'(\tau_t(x))^*J = \Phi(\tau_t(x)) = \tau'_t(\Phi(x)) = J\tau_{-t}(T'(x)^*)J. \]
Thus $T'\tau_t = \tau_{-t} T'$.  However, Proposition~\ref{prop:six}
shows that $T'\tau_t = \tau_t T'$; as $T'$ bijects, it follows that
$\tau_t=\iota$ for all $t$.

So the scaling group of $\G_1$ is trivial; arguing with $(T^{-1})'$
in place of $T'$ shows that the same is true of $\G_2$.  This does not
quite show that $\G_1$ and $\G_2$ are Kac algebras (see \cite[Page~7]{vv})
but it does give us enough that we can now easily follow the proof in
the Kac algebra case, see \cite[Section~5.5]{es}.  Indeed, let
\[ X = (T_p\otimes\iota)(W_1) + (T_{1-p}R_1\otimes\iota)(W_1^*)
\in L^\infty(\G_2) \vnten L^\infty(\hat\G_1), \]
where $W_1$ is the fundamental unitary for $\G_1$.  This makes sense,
as both $T_p$ and $T_{1-p}R_1$ are $*$-homomorphisms.

Then $X$ is unitary.  This follows, as for $x,y\in L^\infty(\G_1)$,
$T_p(x) T_{1-p}(y) = T(x) T(y) p(1-p)=0$.  Thus
\begin{align*} X^*X
&= \big( (T_p\otimes\iota)(W_1^*) + (T_{1-p}R_1\otimes\iota)(W_1) \big)
\big( (T_p\otimes\iota)(W_1) + (T_{1-p}R_1\otimes\iota)(W_1^*) \big) \\
&= (T_p\otimes\iota)(1) + (T_{1-p}R_1\otimes\iota)(1) = 1, \end{align*}
and similarly $XX^*=1$.

As the scaling group is trivial, the familiar formula for the antipode,
\cite[Proposition~8.3]{kv} or \cite[Page~79]{kusvn}, becomes
\[ R_1\big( (\iota\otimes\omega)(W_1) \big) = (\iota\otimes\omega)(W_1^*)
\qquad (\omega\in L^1(\hat\G_1)). \]
Thus, for $\omega\in L^1(\G_2)$ and $\omega'\in L^1(\hat\G_1)$,
as $R_1^2=\iota$,
\begin{align*} \ip{(\omega\otimes\iota)(X)}{\omega'}
&= \ip{T_p((\iota\otimes\omega')W_1)}{\omega}
   + \ip{T_{1-p}R_1((\iota\otimes\omega')(W_1^*))}{\omega} \\
&= \ip{T_p((\iota\otimes\omega')W_1)}{\omega}
   + \ip{T_{1-p}((\iota\otimes\omega')(W_1))}{\omega} \\
&= \ip{W_1}{T_*(\omega)\otimes\omega'}
= \ip{\lambda_1(T_*(\omega))}{\omega'}.
\end{align*}
So the map $L^1(\G_2) \rightarrow C_0(\hat\G_1); \omega\mapsto
\lambda_1(T_*(\omega))$ is a homomorphism, and now a simple calculation
shows that $(\Delta_2\otimes\iota)(X) = X_{13} X_{23}$.

Again, as $S_2 = R_2$, we can turn $L^1(\G_2)$ into a Banach $*$-algebra
for the involution
\[ \ip{x}{\omega^\sharp} = \overline{ \ip{R_2(x)^*}{\omega} }
\qquad (x\in L^\infty(\G_2), \omega\in L^1(\G_2)), \]
compare with Section~\ref{sec:cstar} below.  As $X$ is unitary,
\cite[Proposition~5.2]{kusun} shows that $\lambda_1 T_*$ is a
$*$-homomorphism.  Hence $T_*$ is a $*$-homomorphism.  So, for
$x\in L^\infty(\G_1)$ and $\omega\in L^1(\G_2)$,
\begin{align*} \ip{x}{T_*(\omega)^\sharp} &=
\overline{ \ip{R_1(x)^*}{T_*(\omega)} }
= \overline{ \ip{T(R_1(x))^*}{\omega} } \\
&= \ip{T(x)}{\omega^\sharp}
= \overline{ \ip{R_2(T(x))^*}{\omega} }, \end{align*}
showing that $R_2 T = T R_1$.  Thus $T' = \iota$, so $T'$ is a
$*$-homomorphisms, a contradiction, completing the proof.
\end{proof}

We remark that a corollary of the proof is that, actually, $T'=\iota$
all along!

\subsection{Completely isometric homomorphisms}\label{sec:cccase}

It is increasingly common in non-abelian harmonic analysis to study
objects in the category of operator spaces and completely bounded
maps; see for example the survey \cite{runde}.  It is well-known that the
transpose mapping is the canonical example of an isometric, but not
completely isometric, linear mapping.  So we might suspect that a
complete isometry cannot give rise to a anti-$*$-homomorphism, and this
is indeed the case-- this is well-known, but we include a sketch proof
for completeness.

\begin{theorem}
Let $A$ and $B$ be unital C$^*$-algebras, and let $T:A\rightarrow B$
be completely isometric bijection.  Then $T(1)$ is a unitary, and the
map $A\rightarrow B; a\mapsto T(a)T(1)^*$ is a $*$-homomorphism.
\end{theorem}
\begin{proof}
By Kadison, $T(1)$ is unitary, and $S:a\mapsto T(a)T(1)^*$ is a unital,
completely isometric bijection.  We can now follow \cite[Section~1.3]{bm}
to conclude that $S$ is a $*$-homomorphism, as required.  Indeed, $S$ is
unital and completely contractive, and so is completely positive.  Then
the Stinespring construction allows us to prove the Kadison-Schwarz
inequality: $S(a)^*S(a) \leq S(a^*a)$.  Applying this to $S^{-1}$
as well, and using polarisation, yields the result.
\end{proof}

\begin{theorem}
Let $T_*:L^1(\G_2) \rightarrow L^1(\G_1)$ be a completely isometric algebra
isomorphism.  Then $u = T(1)\in L^\infty(\G_2)$ is a member of the intrinsic
group, and there is a quantum group isomorphism
$\theta:L^\infty(\G_1) \rightarrow L^\infty(\G_2)$ such that
\[ T_*(\omega) = \theta_*(u\omega) \qquad (\omega\in L^\infty(\G_2)). \]
In particular, $\G_1$ is isomorphic to $\G_2$.
\end{theorem}
\begin{proof}
The previous result shows that $\theta(x) = T(x)u^*$ defines a
$*$-homomorphism, and so the result is immediate.
\end{proof}

We remark that if $T_*:L^1(\G_2) \rightarrow L^1(\G_1)$ is isometric,
and completely contractive, then $T_*$ is induced by a quantum group
isomorphism as above.  Indeed, we only need to rule out the possibility
that $\theta:x\mapsto T(x)T(1)^*$ is an anti-$*$-isomorphism.  As $\theta$
is still a complete contraction, the Kadison-Schwarz inequality would yield
that $\theta(aa^*) = \theta(a)^* \theta(a) \leq \theta(a^*a)$, so applying
$\theta^{-1}$ (which is an order-isomorphism) gives $aa^* \leq a^*a$, a
contradiction (unless $L^\infty(\G_1)$ is commutative, in which case
$\theta$ is a homomorphism, as required).

\section{Isometries between measure algebras}\label{sec:cstar}

In this section, we extend our results to isometric algebra isomorphisms
between quantum measure algebras.  We thus start with a survey of
the C$^*$-algebraic theory of locally compact quantum groups.

A \emph{morphism} between two C$^*$-algebras $A$ and $B$ is a non-degenerate
$*$-homomorphism $\phi:A\rightarrow M(B)$ from $A$ to the multiplier
algebra of $B$.  That $\phi$ is non-degenerate is equivalent to $\phi$
extending to a unital, strictly continuous $*$-homomorphism $M(A)\rightarrow
M(B)$.  Thus morphisms can be composed; for further details see
\cite[Appendix~A]{mas}.

Given $\G$ and its fundamental unitary $W$, the space $\{ (\iota\otimes\omega)(W)
: \omega\in \mc B(H)_*\}$ is an algebra $\sigma$-weakly dense in $L^\infty(\G)$.
However, the norm closure turns out to be a C$^*$-algebra, which we shall denote
by $C_0(\G)$.  Then $\Delta$ restricts to give a morphism $C_0(\G)\rightarrow
M(C_0(\G)\otimes C_0(\G))$, and $R$, $\tau_t$ and so forth all restrict
to $C_0(\G)$.  It is possible to define locally compact quantum groups purely
at the C$^*$-algebra level, although the necessary weight theory is more
complicated; see \cite{kv}.

As for $L^1(\G)$, we use $\Delta$ to turn $C_0(\G)^* = M(\G)$ into a Banach
algebra.  In the commutative case, this is the measure algebra of a group,
which justifies the notation.  As $C_0(\G)$ is $\sigma$-weakly dense in
$L^\infty(\G)$, the embedding of $L^1(\G)$ into $M(\G)$ is an isometry;
clearly it is an algebra homomorphism, and actually $L^1(\G)$ becomes an ideal
in $M(\G)$, see \cite[Page~914]{kv}.

Actually, we work in a little generality, and introduce the following
(non-standard) terminology.

\begin{definition}
A \emph{quantum group above $C_0(\G)$} is a triple $(A,\Delta_A,\pi)$
where $A$ is a C$^*$-algebra, $\Delta_A:A\rightarrow M(A\otimes A)$ is
a morphism, coassociative in the sense that $(\iota\otimes\Delta_A)\Delta_A
= (\Delta_A\otimes\iota)\Delta_A$, and $\pi:A\rightarrow C_0(\G)$
is a surjective $*$-homomorphism with $\Delta \pi = (\pi\otimes\pi)\Delta_A$.
Then $\pi^*:M(\G) \rightarrow A^*$ is an algebra homomorphism, and
we make the further requirement that $\pi^*(L^1(\G))$ is an 
essential ideal in $A^*$.  Here \emph{essential} means that
if $\pi^*(\omega)\mu=0$ for all $\omega\in L^1(\G)$, then $\mu=0$,
and similarly with the orders reversed.
\end{definition}

For example, $C_0(\G)$ itself is a quantum group above $C_0(\G)$.
In the cocommutative case, $C_0(\G) = C^*_r(G)$ the reduced group
C$^*$-algebra of a locally compact group $G$, and so $M(\G) = B_r(G)$, the
reduced Fourier-Stieltjes algebra.  We could alternatively study the full
group C$^*$-algebra $C^*(G)$, whose dual is $B(G)$ the Fourier-Stieltjes
algebra.  Then $C^*(G)$ is a quantum group above $C^*_r(G)$.  It turns out
that this example can be generalised to the quantum setting.

We follow \cite{kusun}.  Let $\G$ be a locally compact quantum group,
and let $L^1_\sharp(\hat\G)$ be the collection of $\omega\in L^1(\hat\G)$
such that there is $w^\sharp\in L^1(\hat\G)$ with
\[ \ip{x}{\omega^\sharp} = \overline{ \ip{S(x)^*}{\omega} }
\qquad (x\in D(\hat S)). \]
Then $L^1_\sharp(\hat\G)$ is a $*$-algebra for the involution $\sharp$.
Let $C_u(\G)$ be the universal enveloping $C^*$-algebra of $L^1_\sharp(\hat\G)$,
and let $\hat\lambda_u: L^1_\sharp(\hat\G)\rightarrow C_u(\G)$ be the canonical
homomorphism.  Then $C_u(\G)$ becomes a ``quantum group'' which is very similar
to $C_0(\G)$, the essential difference being that the left and right
invariant weights are no longer faithful.  For us, the important features are:
\begin{itemize}
\item There is a non-degenerate $*$-homomorphism $\Delta_u:C_u(\G)
\rightarrow M(C_u(\G)\otimes C_u(\G))$ which is coassociative;
\item There is a surjective $*$-homomorphism $\pi:C_u(\G) \rightarrow C_0(\G)$
with $\Delta\pi=(\pi\otimes\pi)\Delta_u$.
\end{itemize}
We note here that there are many differences in notation between \cite{kusun}
and that for Kac algebras used in \cite{es}.  We shall continue to follow
\cite{kusun}.  It is shown in \cite[Section~8]{daws} that $L^1(\G)$ is an
essential ideal in $C_u(\G)^*$, and so $C_u(\G)$
is a quantum group above $C_0(\G)$.

In \cite{ks} examples of discrete groups $G$ are given so that there is a
compact quantum group $(A,\Delta_A)$ which ``sits between'' $C^*_r(G)$ and
$C^*(G)$, in the sense that we have proper quotient maps $C^*(G) \rightarrow A
\rightarrow C^*_r(G)$ which intertwine the coproducts.  Then the inclusion maps
$B_r(G) \rightarrow A^* \rightarrow B(G)$ are isometric algebra homomorphisms.
As the Fourier algebra $A(G)$ is an essential ideal in $B(G)$, it follows
that $A$ is a quantum group above $C^*_r(G)$.  Indeed, this argument would
work for any quantum group sitting between $C_0(\G)$ and $C_u(\G)$ (but to our
knowledge, \cite{ks} gives the first example of this phenomena).

\subsection{Quantum group isomorphisms revisited}\label{sec:qgir}

Let $\theta:L^\infty(\G_1) \rightarrow L^\infty(\G_2)$ be a
quantum group isomorphism.  Assuming we have normalised the Haar weights,
$\theta$ will induce an isomorphism between the Hilbert spaces which
intertwines the fundamental unitaries.  Thus $\theta$ will restrict to
give a $*$-isomorphism $C_0(\G_1) \rightarrow C_0(\G_2)$.

Similarly, $\theta$ induces a quantum group isomorphism
$\hat\theta:L^\infty(\hat\G_2) \rightarrow L^\infty(\hat\G_1)$
which satisfies $\hat\theta\lambda_2 = \lambda_1 \theta_*$.
Then $\hat\theta_*$ will restrict to give a $*$-isomorphism between
$L^1_\sharp(\G_1)$ and $L^1_\sharp(\G_2)$.  This will induce a
$*$-isomorphism $\theta_u:C_u(\G_1) \rightarrow C_u(\G_2)$ which
intertwines the coproducts, and satisfies $\pi_2 \theta_u = \theta \pi_1$.

For a quantum group commutant isomorphism $\theta$ we simply compose
$\theta$ with the map $x\mapsto Jx^*J$ to get a quantum group isomorphism
from $\G_1$ to $\G_2'$.  In \cite[Section~4]{kusvn} it is shown that
$(\G_2')\hat{} = \hat\G_2^\op$, where $L^\infty(\hat\G_2^\op)
= L^\infty(\hat\G_2)$ with the opposite coproduct
$\hat\Delta^\op = \sigma \hat\Delta$.  Hence $L^1_\sharp(\hat\G_2^\op)$
agrees with $L^1_\sharp(\hat\G_2)$, but with the reversed product, and
similarly $C_u(\G_2')$ is canonically equal to the opposite C$^*$-algebra to
$C_u(\G_2)$, but has the same coproduct.  Thus, for example, $\theta$
lifts to an anti-$*$-isomorphism $\theta_u:C_u(\G_1) \rightarrow C_u(\G_2)$
which intertwines the coproduct (somewhat as we might hope).

\subsection{Normal extensions}\label{sec:norext}

Let $B$ be a C$^*$-algebra non-degenerately represented on a Hilbert space
$H$.  Let $M = B''$ be the von Neumann algebra generated by $B$.
We can identify the multiplier algebra of $B$ with
\[ M(B) = \{ x\in M : xa,ax \in B \ (a\in B) \}. \]
Let $A$ be a C$^*$-algebra, and consider the enveloping $C^*$-algebra
$A^{**}$.  Let $\phi:A\rightarrow M(B)$ be a morphism.  By the universal
property of $A^{**}$, there is a unique normal $*$-homomorphism $\tilde\phi:
A^{**} \rightarrow B''$ extending $\phi$.
As $\phi$ is non-degenerate, $\tilde\phi$ is unital.
In the special case when $B''=B^{**}$ (say with
$B\subseteq \mc B(H)$ the universal representation) the extension
$\tilde\phi$ is nothing but the second adjoint $\phi^{**}$.

Now let $(A,\Delta_A,\pi)$ be a quantum group above $C_0(\G)$, and let
$A\subseteq\mc B(H)$ be the universal representation, so that both $A\otimes A$
(the spacial C$^*$-tensor product) and $A^{**}\vnten A^{**}$ are
subalgebras of $\mc B(H\otimes H)$.  We can hence form the extension
$\tilde\Delta_A:A^{**}\rightarrow A^{**}\vnten A^{**}$.  Notice that
then $(A^{**},\tilde\Delta_A)$ becomes a Hopf-von Neumann algebra.

Similarly, we form $\tilde\pi:A^{**} \rightarrow L^\infty(\G)$.
The preadjoint of this map is simply the embedding $\tilde\pi_*:L^1(\G)
\rightarrow A^*$, which is the composition of the isometry $L^1(\G)
\rightarrow C_0(\G)^*$ with the isometry $\pi^*:C_0(\G)^*\rightarrow A^*$.
Let $\supp\tilde\pi$ be the support projection of $\tilde\pi$, so
$\supp\tilde\pi\in A^{**}$ is the unique central projection with, for
$x\in A^{**}$, $x\supp\tilde\pi=0$ if and only if $\tilde\pi(x)=0$.
Then
\[ \tilde\pi_*(L^1(\G))^\perp = \{ x\in A^{**} : \ip{x}{\tilde\pi_*(\omega)}
=0 \ (\omega\in L^1(\G)) \} = \ker\tilde\pi
= (1-\supp\tilde\pi)A^{**}. \]
It follows that
\begin{align*} \tilde\pi_*(L^1(\G)) &=
\{ \mu\in A^* : \ip{x}{\mu}=0 \ (x\in(1-\supp\tilde\pi)A^{**}) \} \\
&= \{ \mu\in A^* : (1-\supp\tilde\pi)\mu=0 \}
= (\supp\tilde\pi) A^*.
\end{align*}

Temporarily, let $\Delta_0$ be the coproduct on $C_0(\G)$, and let
$\Delta_\infty$ be the coproduct on $L^\infty(\G)$.  Identifying $M(C_0(\G)
\otimes C_0(\G))$ with a subalgebra of $L^\infty(\G)\vnten L^\infty(\G)$,
we see that $\Delta_\infty$ extends $\Delta_0$.  It is easy to verify
that as $(\pi\otimes\pi)\Delta_A = \Delta_0 \pi$, also
$(\tilde\pi\otimes\tilde\pi)\tilde\Delta_A = \Delta_\infty \tilde\pi$.
We shall use this, and similar relations, without comment in the next section.

We remark that we could work with a more general notion
of a quantum group above $C_0(\G)$.  Indeed, suppose that $(A,\Delta_A)$ is
a C$^*$-bialgebra, and that $\pi:A\rightarrow L^\infty(\G)$ is a
$*$-homomorphism with $\sigma$-weak dense range.  Then $\pi\otimes \pi$
is a $*$-homomorphism $A\otimes A\rightarrow L^\infty(\G)\otimes L^\infty(\G)
\subseteq L^\infty(\G)\vnten L^\infty(\G)$, and so, by taking a normal
extension, we have a $*$-homomorphism $\pi:M(A\otimes A)\rightarrow 
L^\infty(\G)\vnten L^\infty(\G)$.  We can thus make sense of the
requirement that $(\pi\otimes\pi)\Delta_A = \Delta \pi$.  Then $\pi^*$
restricted to $L^1(\G)$ gives a homomorphism $L^1(\G)\rightarrow A^*$ (which
is an isometry, as $\pi$ has $\sigma$-weakly dense range, and using Kaplansky
Density).  We again insist that $\pi^*(L^1(\G))$ is an essential ideal in $A^*$.
A careful examination of the following proofs show that they would all work in
this more general setting; but in the absence of any examples, we do not make
this a formal definition.

\subsection{Isometries of duals of quantum groups}

For $i=1,2$ let $(A_i,\Delta_{A_i},\pi_i)$ be a quantum group above $C_0(\G_i)$.
Let $T_*:A_2^* \rightarrow A_1^*$ be an isometric algebra isomorphism, and
set $T=(T_*)^*:A_1^{**} \rightarrow A_2^{**}$.  The following is now proved
in an entirely analogous way to the arguments in Section~\ref{sec:ell1case}.

\begin{lemma}\label{lem:four}
$T(1)$ is a unitary element of $A_2^{**}$ which is a member of the intrinsic
group.  The map $T_{1*}:A_1^{**}\rightarrow A_2^{**}; \omega\mapsto
T_*(T(1)^*\omega)$ is an isometric algebra isomorphism.
\end{lemma}

Again, we find that $T_1 = (T_{1*})^*$ is a Jordan homomorphism.
We now show (in a similar, but more general, fashion to the arguments in
\cite[Section~5.6]{es}) a link between $\tilde\pi$ and the order properties
of $A^{**}$, for $(A,\Delta_A,\pi)$ a quantum group above $C_0(\G)$.

\begin{proposition}\label{prop:seven}
Let $\G$ be a locally compact quantum group, and let
$(A,\Delta_A)$ a quantum group above $C_0(\G)$.  Let
\[ \mc Q = \{ Q\in A^{**} : Q\text{ is a projection, } Q\not=1,
\tilde\Delta_A(Q)\leq Q\otimes Q \}. \]
Then $\mc Q$ has a maximal element, which is $1-\supp\tilde\pi$.
\end{proposition}
\begin{proof}
Let $e=1-\supp\tilde\pi$, so as in Section~\ref{sec:norext} above,
$e A^{**} = \ker\tilde\pi$ and $\tilde\pi_*(L^1(\G)) = (1-e)A^*$.
Let $\mu,\mu'\in A^*$, so there are $\omega,\omega'\in L^1(\G)$ with
$\tilde\pi_*(\omega) = (1-e)\mu$ and $\tilde\pi^*(\omega')=(1-e)\mu'$.
Then
\begin{align*} \ip{\tilde\Delta_A(e)(e\otimes e)}{\mu\otimes\mu'}
&= \ip{e}{(e\mu)(e\mu')}
= \ip{e}{(\mu-\tilde\pi_*(\omega))(\mu'-\tilde\pi_*(\omega'))} \\
&= \ip{e}{\mu\mu'} + \ip{e}{\tilde\pi_*(\omega\omega') - \mu\tilde\pi_*(\omega')
- \tilde\pi_*(\omega)\mu'} \\
&= \ip{e}{\mu\mu'}
= \ip{\tilde\Delta_A(e)}{\mu\otimes\mu'},
\end{align*}
as $\tilde\pi_*(L^1(\G))$ is an ideal in $A^*$.
It follows that $\tilde\Delta_A(e)(e\otimes e) = \tilde\Delta_A(e)$,
and so $\tilde\Delta_A(e) \leq e\otimes e$.  Thus $e\in\mc Q$.

Now let $Q\in\mc Q$, so that $\Delta \tilde\pi(Q) = (\Delta\pi)\tilde{\ }(Q)
= ((\pi\otimes\pi)\Delta_u)\tilde{\ }(Q)
= (\tilde\pi\otimes\tilde\pi)\tilde\Delta_u(Q)
\leq \tilde\pi(Q) \otimes \tilde\pi(Q) \leq 1 \otimes \tilde\pi(Q)$.
By \cite[Lemma~6.4]{kv}, this can only occur when $\tilde\pi(Q)=0$ or $1$.

If $\tilde\pi(Q)=1$, then $Q \geq \supp\tilde\pi$, and so $Q+e\geq 1$.
Thus $\tilde\Delta_u(Q) + \tilde\Delta_u(e) \geq 1\otimes 1$, but as
$Q,e\in \mc Q$, it follows that
\[ 1 \otimes 1 \leq Q\otimes Q + e\otimes e. \]
Thus also
\[ (1-Q)\otimes(1-e) \leq \big((1-Q)\otimes(1-e)\big)
\big( Q\otimes Q + e\otimes e\big) \big((1-Q)\otimes(1-e)\big) = 0, \]
and so $Q=1$ or $e=1$, a contradiction.
Thus $\tilde\pi(Q)=0$, showing that $Q\leq e$ as required.
\end{proof}

\begin{proposition}\label{prop:eight}
With $T_*,T,T_1$ as above, we have that:
\begin{enumerate}
\item\label{prop:eightone} $T_1(1-\supp\tilde\pi_1) = 1-\supp\tilde\pi_2$.
\item\label{prop:eighttwo} $T(\ker\tilde\pi_1) = \ker\tilde\pi_2$.
\item\label{prop:eightthree} $T_*(\tilde\pi_{2,*}(L^1(\G_2)))
   = \tilde\pi_{1,*}(L^1(\G_1))$.
\end{enumerate}
\end{proposition}
\begin{proof}
For $i=1,2$, form $\mc Q_i$ for $A_i$ as in Proposition~\ref{prop:seven}.
We claim that $T_1$ gives a bijection $\mc Q_1$ to $\mc Q_2$.  Let
$Q\in \mc Q_1$, so as $T_1$ is Jordan homomorphism, $T_1(Q)$ is a projection
which is not equal to $1$ (as $T_1(1)=1$ and $T_1$ bijects).
For $\mu,\mu'\in A_2^*$ positive, we see that
\begin{align*} \ip{\tilde\Delta_{A_2}(T_1(Q))}{\mu\otimes\mu'}
&= \ip{\Delta_{A_1}(Q)}{T_{1*}(\mu)\otimes T_{1*}(\mu')} \\
&\leq \ip{Q\otimes Q}{T_{1*}(\mu)\otimes T_{1*}(\mu')}
= \ip{T_1(Q)\otimes T_1(Q)}{\mu\otimes\mu'}. \end{align*}
It follows that $\tilde\Delta_{A_2}(T_1(Q)) \leq T_1(Q)\otimes T_1(Q)$, and so
$T_1(\mc Q_1) \subseteq \mc Q_2$.  Applying the same argument to $T_1^{-1}$
yields that $T_1(\mc Q_1) \supseteq \mc Q_2$, giving the claim.
As $T_1$ preserves the order, and $1-\supp\tilde\pi_i$ is the maximal
element of $\mc Q_i$, it follows that $T_1(1-\supp\tilde\pi_1)
= 1-\supp\tilde\pi_2$ showing (\ref{prop:eightone}).

For $i=1,2$, we know that $x\in\ker\tilde\pi_i$ if and only if
$x \supp\tilde\pi_i = 0$.  For $x\in A_1^{**}$, as $T_1$ is a Jordan
homomorphism, we see that
\begin{align*} 2 T_1(x) \supp\tilde\pi_2
&= (\supp\tilde\pi_2)T_1(x) + T_1(x)\supp\tilde\pi_2
= T_1(\supp\tilde\pi_1) T_1(x) + T_1(x) T_1(\supp\tilde\pi_1) \\
&= T_1((\supp\tilde\pi_2) x+x \supp\tilde\pi_2) = 2T_1(x \supp\tilde\pi_2),
\end{align*}
using (\ref{prop:eightone}).  Thus $T_1(\ker\tilde\pi_1)
= \ker\tilde\pi_2$.  As $\ker\tilde\pi_1$ is an ideal, and
$T(1)$ is unitary, it follows that $\ker\tilde\pi_1 T(1) = \ker\tilde\pi_1$,
and so $T(\ker\tilde\pi_1) = T_1(\ker\tilde\pi_1 T(1)) = \ker\tilde\pi_2$
showing (\ref{prop:eighttwo}).

As in Section~\ref{sec:norext} above, we have that
$\tilde\pi_{i,*}(L^1(\G_i)) = {}^\perp(\ker\tilde\pi_i)$, for $i=1,2$.
Hence (\ref{prop:eightthree}) follows immediately from (\ref{prop:eighttwo}).
\end{proof}

For $i=1,2$ we have that $L^1(\G_i) \subseteq A_i^*$ isometrically,
and so the restriction of $T$ yields an isometric algebra homomorphism
$T_r:L^1(\G_2) \rightarrow L^1(\G_1)$.  We have already characterised such
maps, and we next bootstrap this to determine the structure of $T_*$ on all
of $A_2^*$.  For the moment, we restrict attention to the cases when
$A_i=C_0(\G_i)$ for $i=1,2$, or $A_i=C_u(\G_i)$, for $i=1,2$.  In the next
section we use quantum group duality to say something about the general case.

Given a quantum group (commutant) isomorphism $\theta:L^\infty(\G_1)
\rightarrow L^\infty(\G_2)$, we recall from Section~\ref{sec:qgir}
that $\theta$ restricts to a (anti-) $*$-isomorphism $\theta:C_0(\G_1)
\rightarrow C_0(\G_2)$, and lifts to a (anti-) $*$-isomorphism
$\theta_u:C_u(\G_1) \rightarrow C_u(\G_2)$.  In the following, we call
such a map ``associated''.

\begin{theorem}\label{thm:two}
Let $\G_1$ and $\G_2$ be locally compact quantum groups.
Suppose that either $A_1=C_0(\G_1), A_2=C_0(\G_2)$, or $A_1=C_u(\G_1)$,
$A_2 = C_u(\G_2)$.
Let $T_*:A_2^* \rightarrow A_1^*$ be a bijective isometric algebra
homomorphism, and set $T=(T_*)^*$.  Then $v=T(1)$ and $u=T_r^*(1)$ are
in the intrinsic groups of $A_2^{**}$ and $L^\infty(\G_2)$, respectively.
There is either:
\begin{enumerate}
\item A quantum group isomorphism $\theta:L^\infty(\G_1) \rightarrow
L^\infty(\G_2)$ and associated $*$-isomorphism
$\theta_0:A_1\rightarrow A_2$ which intertwines the coproducts; or
\item A quantum group commutant isomorphism $\theta:L^\infty(\G_1) \rightarrow
L^\infty(\G_2)$ and associated anti-$*$-isomorphism
$\theta_0:A_1\rightarrow A_2$ which intertwines the coproducts.
\end{enumerate}
In either case, for $\omega\in M_u(\G_2)$ and $\omega'\in L^1(\G_2)$,
\[ T_*(\omega) = \theta_0^*(v \omega), \qquad
T_r(\omega') = \theta_*(u\omega'). \]
\end{theorem}
\begin{proof}
By previous work, $u=T_r^*(1)$ is a member of the intrinsic group of
$L^\infty(\G_2)$, and the map
$\theta : L^\infty(\G_1)\rightarrow L^\infty(\G_2); x \mapsto T_r^*(x)u^*$ is
either a normal $*$-isomorphism, or a normal anti-$*$-isomorphism, which in
either case intertwines the coproduct.

Suppose we are in the first case, where $\theta$ is a $*$-isomorphism.
Then we have an associated $*$-isomorphism $\theta_0:A_1\rightarrow A_2$
which intertwines the coproducts, and which satisfies $\pi_2 \theta_0
= \theta \pi_1$.  Taking adjoints gives that $\theta_0^* \tilde\pi_{2,*}
= \tilde\pi_{1,*} \theta_*$.

For $\omega\in L^1(\G_2)$, we have that $\theta_*(\omega) = T_r(u^*\omega)$.
Also, $T_r$ is constructed so that $T_* \tilde\pi_{2,*} = \tilde\pi_{1,*} T_r$.
Thus
\[ T_*\big( \tilde\pi_{2,*}(u^*\omega) \big)
= \tilde\pi_{1,*}\big( T_r(u^*\omega) \big)
= \tilde\pi_{1,*}\big( \theta_*(\omega) \big)
= \theta_0^*\big(\tilde\pi_{2,*}(\omega)\big)
\qquad (\omega\in L^1(\G_2)). \]

Recall that $v = T(1) \in A_2^{**}$ is also unitary.  Then $u = T_r^*(1)
= T_r^* \tilde\pi_1(1) = \tilde\pi_2 T(1) = \tilde\pi_2(v)$.
A simple calculation shows that $v \tilde\pi_{2,*}(\omega)
= \tilde\pi_{2,*}(u\omega)$ for $\omega\in L^1(\G_2)$.

Let $\mu\in A_2^*$ and $\omega\in L^1(\G_2)$, so we can find
$\omega'\in L^1(\G_2)$ with $\tilde\pi_{2,*}(\omega') =
\tilde\pi_{2,*}(\omega) \mu$.  Then
\begin{align*}
T_*\big( \tilde\pi_{2,*}(\omega) \big) T_*(\mu)
&= T_*\big( \tilde\pi_{2,*}(\omega') \big)
= \theta_0^*\big(\tilde\pi_{2,*}(u\omega')\big)
= \theta_0^*\big(v\tilde\pi_{2,*}(\omega')\big)
= \theta_0^*\big(v\big(\tilde\pi_{2,*}(\omega)\mu\big)\big) \\
&= \theta_0^*\big( (v\tilde\pi_{2,*}(\omega)) (v\mu) \big)
= \theta_0^*(v\tilde\pi_{2,*}(\omega)) \theta_0^*(v\mu)
= T_*\big( \tilde\pi_{2,*}(\omega) \big) \theta_0^*(v\mu).
\end{align*}
Recall that, from the hypothesis, $\tilde\pi_{1,*}(L^1(\G_2))$ is an essential
ideal in $A_1^*$.  As $T_*$ bijects $\tilde\pi_{2,*}(L^1(\G_2))$
to $\tilde\pi_{1,*}(L^1(\G_2))$, we see that
\[ T_*(\mu) = \theta_0^*(v\mu) \qquad (\mu\in A_2^*), \]
as claimed.

The other case, when $\theta$ is a quantum group commutant isomorphism,
is entirely analogous.
\end{proof}

The previous theorem needs a characterisation of the intrinsic group of
$A^{**}$, for $A$ a quantum group above $C_0(\G)$.  The following results
show that it is enough to know the intrinsic group of $L^\infty(\G)$.

\begin{lemma}
Let $A$ be a Banach algebra, and let $I\subseteq A$ be a closed ideal.
Let $\Phi_I$ be the character space $I$, and let $X$ be the collection
of characters on $A$ which do not restrict to the zero functional on $I$.
Then restriction of linear functionals gives a bijection from $X$ to $\Phi_I$.
\end{lemma}
\begin{proof}
Let $f,g\in X$ induce the same (non-zero) character on $I$.  Pick $a_0\in I$
with $f(a_0) = g(a_0)=1$.  Then, for $a\in A$, we see that
$f(a) = f(a)f(a_0) = f(aa_0) = g(aa_0) = g(a)g(a_0) = g(a)$, using that
$aa_0\in I$.  Thus $f=g$, so the restriction map is injective.

Now let $u\in\Phi_I$, and pick $a_0\in I$ with $u(a_0)=1$.  Define $f\in A^*$
by $f(a) = u(aa_0)$ for each $a\in A$.  Then, for $a,b\in A$,
\begin{align*} f(ab) &= u(aba_0) = u(a_0)u(aba_0) = u(a_0a ba_0)
= u(a_0a) u(ba_0) = u(a_0a) u(a_0) f(b) \\
&= u(a_0aa_0) f(b) = u(a_0) u(aa_0) f(b) = f(a) f(b). \end{align*}
So $f$ is a character on $A$.  For $a\in I$, also $f(a) = u(aa_0) = u(a)u(a_0)
= u(a)$, and so $f\in X$ and $f$ restricts to $u$.  Thus the restriction map
is a bijection.
\end{proof}

The following should be compared with \cite[Theorem~1]{wal}
where Walter shows this in the cocommutative case.

\begin{theorem}\label{thm:five}
Let $(A,\Delta_A,\pi)$ be a quantum group above $C_0(\G)$.  For a character
$u$ on $A^*$, the following are equivalent:
\begin{enumerate}
\item $u$ is a member of the intrinsic group of $A^{**}$;
\item $u$ is invertible in $A^{**}$;
\item $\tilde\pi(u)\not=0$, that is, $u$ does not induce the zero
functional on $\tilde\pi_*(L^1(\G))$.
\end{enumerate}
Moreover, $\tilde\pi:A^{**}\rightarrow L^\infty(\G)$ restricts to a bijection
between the intrinsic groups of $A^{**}$ and $L^\infty(\G)$.
\end{theorem}
\begin{proof}
Let $Y$ be the intrinsic group of $L^\infty(\G)$, which by
Theorem~\ref{prop:one} is the character space of $L^1(\G)$.
Let $X_1$ be the intrinsic group of $A^{**}$, let $X_2$ be the
collection of invertible characters, and let $X_3$ be the collection
of characters not sent to zero by $\tilde\pi$.  If $u\in X_2$ then
$1 = \tilde\pi(1) = \tilde\pi(u u^{-1}) = \tilde\pi(u) \tilde\pi(u^{-1})$,
showing that $\tilde\pi(u)\not=0$ and hence $u\in X_3$.  Thus
$X_1 \subseteq X_2 \subseteq X_3$.  By the lemma, $\tilde\pi$ restricts
to a bijection between $X_3$ and $Y$.

Let $u\in X_3$, so by Theorem~\ref{prop:one}, $\tilde\Delta_A(u) = u \otimes u$.
As $\tilde\Delta_A$ is a $*$-homomorphism, also $u^*u$ is a character.
As $\tilde\pi(u)\in L^\infty(\G)$ is a (non-zero) character, it is unitary, and
so $1 = \tilde\pi(u)^* \tilde\pi(u) = \tilde\pi(u^*u)$.  Thus $u^*u\in X_3$,
and as $\tilde\pi$ injects on $X_3$, and $1\in X_3$, we conclude that
$u^*u=1$.  Similarly, $uu^*=1$.  Thus $u$ is a member of the intrinsic group
of $A^{**}$, that is, $u\in X_1$.  We hence have the required equalities
$X_1 = X_2 = X_3$.
\end{proof}

In special cases, we can say more.

\begin{proposition}
The intrinsic group of $C_u(\G)^{**}$, respectively $C_0(\G)^{**}$,
is a subgroup of the unitary group of $M(C_u(\G))$, respectively $M(C_0(\G))$.
\end{proposition}
\begin{proof}
Let $x\in C_u(\G)^{**}$ be a member of the intrinsic group, and set $y
= \tilde\pi(x) \in L^\infty(\G)$.  By Theorem~\ref{prop:one}, we have that
$y$ is unitary, and $y\in M(C_0(\G))$.  Thus, in the language of
\cite[Proposition~6.6]{kusun}, $y$ is a unitary corepresentation of
$C_0(\G)$ on $\mathbb C$, and so there is $x_0 \in M(C_u(\G))$ with
$\pi(x_0) = y$ and $\Delta_u(x_0)=x_0\otimes x_0$.  By uniqueness
(from the previous theorem) we must have
that $x_0 = x$, treating $M(C_u(\G))$ as a subalgebra of $C_u(\G)^{**}$.

Now let $x\in C_0(\G)^{**}$ be a member of the intrinsic group.  Again,
$\tilde\pi(x) = y \in M(C_0(\G))$, so let $x_0$ be the image of $y$ under
the embedding $M(C_0(\G)) \rightarrow C_0(\G)^{**}$.  Thus $x_0$ is a member
of the intrinsic group of $C_0(\G)^{**}$ and $\tilde\pi(x_0)=\tilde\pi(x)$,
so again by uniqueness, we conclude that $x_0=x$.
\end{proof}

\subsection{The picture under duality}

We now show that by using the duality theory of locally compact
quantum groups, we can handle the more general situation; this
also gives results more reminiscent of those for Kac algebras, see
\cite[Section~5.6]{es}.

Let $(A,\Delta_A,\pi)$ be a quantum group above $C_0(\G)$.
As $L^1(\G)$ is an essential ideal in $A^*$, each member of $A^*$
induces a (completely bounded) multiplier (or centraliser) of $L^1(\G)$.
Let us introduce the notation that given $\mu\in A^*$, we have maps
$L_\mu, R_\mu:L^1(\G) \rightarrow L^1(\G)$ with
\[ \tilde\pi_* L_\mu(\omega) = \mu \tilde\pi_*(\omega), \quad
\tilde\pi_* R_\mu(\omega) = \tilde\pi_*(\omega) \mu
\qquad (\omega\in L^1(\G)). \]
Let us denote by $M_{cb}(L^1(\G))$ the algebra of completely bounded
multipliers of $L^1(\G)$.  In \cite[Theorem~8.9]{daws}, a homomorphism
$\Lambda:M_{cb}(L^1(\G)) \rightarrow C^b(\hat\G)$ was constructed
(and a more general construction, with one-sided multipliers, is
given in \cite{jnr}).  We hence find a map, which we shall continue
to denote by $\Lambda$, from $A^*$ to $C^b(\hat\G)$, which is
uniquely determined by the properties that
\[ \Lambda(\mu)\lambda(\omega) = \lambda(L_\mu(\omega)), \quad
\lambda(\omega)\Lambda(\mu) = \lambda(R_\mu(\omega))
\qquad (\mu\in A^*, \omega\in L^1(\G)). \]

An important link between multipliers and the antipode is established
in \cite{daws1}.  In particular, given $\mu\in A^*$, define an associated
left multiplier $L_\mu^\dagger$ by
\[ L_\mu^\dagger(\omega) = L_\mu(\omega^*)^* \quad\text{so}\quad
\tilde\pi_* L_\mu^\dagger(\omega)
= \big( \mu \tilde\pi_*(\omega^*) \big)^*
= \mu^* \tilde\pi_*(\omega), \]
that is, $L_\mu^\dagger = L_{\mu^*}$.
(Recall that $\omega^*$ is the normal functional $x\mapsto
\overline{ \ip{x^*}{\omega} }$ so this calculation follows immediately
from the fact that $\tilde\Delta_A$ and $\tilde\pi$ are $*$-homomorphisms).
Then \cite[Theorem~5.9]{daws1} shows that $\Lambda(\mu^*) \in D(\hat S)^*$
and $\Lambda(\mu) = \hat S(\Lambda(\mu^*)^*)$.

For $\lambda:L^1(\G) \rightarrow C_0(\G)$ (which $\Lambda$ extends) we can
see this directly.  Recall (see \cite[Proposition~8.3]{kv}) that
$\hat S((\iota\otimes\omega)(\hat W)) = (\iota\otimes\omega)(\hat W^*)$.
As $\hat W = \sigma W^*\sigma$, we see that $\lambda(\omega)
= (\omega\otimes\iota)(W) = \hat S((\omega\otimes\iota)(W^*))
= \hat S(\lambda(\omega^*)^*)$.

\begin{lemma}\label{lem:nine}
Let $u\in L^\infty(\G)$ be a member of the intrinsic group.  For
$x\in L^\infty(\hat\G)$, let $\hat\gamma_u(x) = uxu^*$.  Then $\hat\gamma_u$
is a $*$-automorphism of $L^\infty(\hat\G)$, which restricts to a $*$-automorphism
of $C_0(\hat\G)$.  Furthermore, $\hat\gamma_u\lambda(\omega) =
\lambda(u\omega)$ for $\omega\in L^1(\G)$.
\end{lemma}
\begin{proof}
We have that $\Delta(u)=u\otimes u$, so $W^*(1\otimes u)W = u\otimes u$.
Using that $W$ and $u$ are unitary, it follows that $(1\otimes u)W(1\otimes u^*)
= W(u\otimes 1)$.  Then, for $\omega\in L^1(\G)$,
\[ \hat\gamma_u\lambda(\omega) = u(\omega\otimes\iota)(W)u^*
= (\omega\otimes\iota)\big( (1\otimes u)W(1\otimes u^*) \big)
= (\omega\otimes\iota)\big( W(u\otimes 1) \big)
= \lambda(u\omega), \]
as claimed.  By density, it follows that $\hat\gamma_u$ is a self-map of
$C_0(\hat\G)$, which clearly has the inverse $\hat\gamma_{u^*}$.
As $\hat\gamma_u$ is normal, it follows that $\hat\gamma_u$ is also an
automorphism of $L^\infty(\hat\G)$.
\end{proof}

For the construction of $\hat\theta$ in the following, we again
refer the reader to Section~\ref{sec:qgir}.

\begin{theorem}\label{thm:three}
For $i=1,2$, let $\G_i$ be a locally compact quantum groups, and let
$(A_i,\Delta_{A_i},\pi_i)$ be a quantum group above $C_0(\G_i)$.
Let $T_*:A_2^* \rightarrow A_1^*$ be a bijective isometric algebra
homomorphism, and set $T=(T_*)^*$.
Then $v=T(1)$ and $u = \tilde\pi_2(v)$ are members of the
intrinsic groups of $A_2^{**}$ and $L^\infty(\G_2)$, respectively.
Then either:
\begin{enumerate}
\item\label{thm:threeone} There is a quantum group isomorphism
$\theta:L^\infty(\G_1) \rightarrow L^\infty(\G_2)$, leading to a quantum
group isomorphism $\hat\theta:L^\infty(\hat\G_2)\rightarrow
L^\infty(\hat\G_1)$, with
\[ \Lambda_1 T_* = \hat\theta \hat\gamma_u \Lambda_2. \]
\item\label{thm:threetwo}
There is a quantum group commutant isomorphism
$\theta:L^\infty(\G_1) \rightarrow L^\infty(\G_2)$, leading to a quantum
group isomorphism $\hat\theta:L^\infty(\hat\G_2^\op)\rightarrow
L^\infty(\hat\G_1)$, with
\[ \Lambda_1 T_* = \hat\theta \hat R_2 \hat S_2^{-1} \hat\gamma_u \Lambda_2. \]
\end{enumerate}
In particular, $\G_1$ is isomorphic to either $\G_2$ or $\G_2'$.
\end{theorem}
\begin{proof}
In this more general situation, the proof of Theorem~\ref{thm:two},
and Theorem~\ref{thm:five}, still give
the facts about $v$ and $u$, and yields $\theta$ such that
\[ T_*\big( \tilde\pi_{2,*}(u^*\omega) \big)
= \tilde\pi_{1,*}\big( \theta_*(\omega) \big)
\qquad (\omega\in L^1(\G_2)). \]

Suppose first that $\theta$ is a quantum group isomorphism.
Let $\hat\theta:L^\infty(\hat\G_2) \rightarrow L^\infty(\hat\G_1)$ be
the quantum group isomorphism induced by $\theta$, which satisfies
$\lambda_1 \theta_* = \hat\theta\lambda_2$.

Let $\omega\in L^1(\G_2)$ and $\mu\in A_2^*$.  There is $\omega'\in L^1(\G_2)$
with $\mu\tilde\pi_{2,*}(\omega) = \tilde\pi_{2,*}(\omega')$.  Then
\begin{align*} \Lambda_1(T_*(\mu)) \Lambda_1(T_*(\tilde\pi_{2,*}(\omega)))
&= \Lambda_1(T_*(\tilde\pi_{2,*}(\omega')))
= \lambda_1(\theta_*(u\omega'))
= \hat\theta(\lambda_2(u\omega'))
= \hat\theta(\hat\gamma_u(\lambda_2(\omega'))) \\
&= \hat\theta\hat\gamma_u\big( \Lambda_2(\mu) \lambda_2(\omega) \big)
= \hat\theta\hat\gamma_u\Lambda_2(\mu) \Lambda_1(T_*(\tilde\pi_{2,*}(\omega))),
\end{align*}
using that, similarly, $\Lambda_1(T_*(\tilde\pi_{2,*}(\omega)))
= \hat\theta(\hat\gamma_u(\lambda_2(\omega)))$.  As the set
\[ \{ \hat\gamma_u(\lambda_2(\omega)) : \omega\in L^1(\G_2) \} \]
is norm dense in $C_0(\hat\G_2)$, working in $M(C_0(\hat\G_2))$, we conclude
that
\[ \Lambda_1T_*(\mu) = \hat\theta \hat\gamma_u \Lambda_2(\mu)
\qquad (\mu\in A_2^*), \]
as required.

In the case when $\theta$ is a quantum group commutant isomorphism,
define $\Phi:L^\infty(\G_2) \rightarrow L^\infty(\G_2'); x\mapsto Jx^*J$,
and set $\theta' = \Phi\theta:L^\infty(\G_1) \rightarrow L^\infty(\G_2')$,
which is a quantum group isomorphism.
As in Section~\ref{sec:qgir} we find a quantum group isomorphism
$\hat\theta' : L^\infty((\G_2')\hat{\ }) \rightarrow L^\infty(\hat\G_1)$.
As $(\G_2')\hat{\ } = (\hat\G_2)^\op$, this gives a normal $*$-isomorphism
$\hat\theta:L^\infty(\hat\G_2) \rightarrow L^\infty(\hat\G_1)$ with
$\hat\Delta_1\hat\theta = \sigma(\hat\theta\otimes\hat\theta)\hat\Delta_2$.
Then $\hat\theta \lambda_2' = \lambda_1\theta'_* = \lambda_1\theta_*\Phi_*$.

We now calculate $\lambda_2'\Phi_*^{-1}$.
Let $\xi,\eta,\alpha,\beta\in L^2(\G)$, so
\begin{align*} \big( \lambda_2'\Phi_*^{-1}(\omega_{\xi,\eta})
   \alpha \big| \beta \big)
&= \big( \lambda_2'(\omega_{J\eta,J\xi}) \alpha \big| \beta \big)
= \big( W'_2(J\eta\otimes\alpha) \big| J\xi\otimes\beta \big) \\
&= \big( (J\otimes J)W_2(\eta\otimes J\alpha) \big| J\xi\otimes\beta \big)
=  \big( W_2^*(\xi\otimes J\beta) \big| \eta\otimes J\alpha\big) \\
&= \big( (\omega_{\xi,\eta}\otimes\iota)(W_2^*) J\beta \big| J\alpha \big)
= \big( J((\omega_{\xi,\eta}\otimes\iota)(W_2^*))^* J\alpha \big| \beta \big),
\end{align*}
using that $W_2' = (J\otimes J)W_2(J\otimes J)$.
With reference to the discussion before Lemma~\ref{lem:nine}, we see that
\[ \lambda_2'\Phi_*^{-1}(\omega) = \hat R_2\big( (\omega\otimes\iota)(W_2^*) \big)
= \hat R_2\big( \lambda_2(\omega^*)^* \big)
= \hat R_2 \hat S_2^{-1} \lambda_2(\omega) \qquad (\omega\in L^1(\G_2)). \]
In particular,
\[ \lambda_1 \theta_* = \hat\theta \hat R_2 \hat S_2^{-1} \lambda_2. \]

Finally, we follow the previous argument through.  So let $\omega,\omega'
\in L^1(\G_2)$ and $\mu\in A_2^*$ with $\mu \tilde\pi_{2,*}(\omega) =
\tilde\pi_{2,*}(\omega')$.  Then
\begin{align*} \Lambda_1(T_*(\mu)) \Lambda_1(T_*(\tilde\pi_{2,*}(\omega)))
&= \lambda_1(\theta_*(u\omega'))
= \hat\theta \hat R_2 \hat S_2^{-1} \hat\gamma_u \lambda_2(\omega')
= \hat\theta \hat R_2 \hat S_2^{-1} \hat\gamma_u
   \big( \Lambda_2(\mu) \lambda_2(\omega) \big) \\
&= \big(\hat\theta \hat R_2 \hat S_2^{-1} \hat\gamma_u \Lambda_2(\mu)\big)
\Lambda_1(T_*(\tilde\pi_{2,*}(\omega))),
\end{align*}
using that $\hat R_2 \hat S_2^{-1}$ is a homomorphism on $D(S_2^{-1})$.
This completes the proof.
\end{proof}

\small
\noindent
Matthew Daws\\
School of Mathematics,\\
University of Leeds,\\
LEEDS LS2 9JT\\
United Kingdom\\
Email: \texttt{matt.daws@cantab.net}

\smallskip

\noindent
Hung Le Pham\\
School of Mathematics, Statistics, and Operations Research,\\
Victoria University of Wellington,\\
Wellington 6012\\
New Zealand\\
Email: \texttt{hung.pham@vuw.ac.nz}

\end{document}